\documentclass[reqno]{amsart}

\usepackage[utf8]{inputenc}%
\usepackage[T1]{fontenc}%
\usepackage{lmodern}%

\usepackage{amssymb}%
\usepackage{graphicx}%
\usepackage[shortlabels]{enumitem}%
\usepackage{multirow}%

\usepackage{cite}%

\newcommand{\cA}{\ensuremath{\mathcal{A}}}%
\newcommand{\cB}{\ensuremath{\mathcal{B}}}%
\newcommand{\cE}{\ensuremath{\mathcal{E}}}%
\newcommand{\cK}{\ensuremath{\mathcal{K}}}%
\newcommand{\cL}{\ensuremath{\mathcal{L}}}%
\newcommand{\cM}{\ensuremath{\mathcal{M}}}%

\newcommand{\bbC}{\ensuremath{\mathbb{C}}}%
\newcommand{\bbE}{\ensuremath{\mathbb{E}}}%
\newcommand{\bbP}{\ensuremath{\mathbb{P}}}%
\newcommand{\bbR}{\ensuremath{\mathbb{R}}}%

\newcommand{\bOne}{\ensuremath{\mathbf{1}}}%

\newcommand{\TV}{\textsc{tv}}%

\theoremstyle{plain}%
\newtheorem{theorem}{Theorem}%
\newtheorem{lemma}[theorem]{Lemma}%
\newtheorem{proposition}[theorem]{Proposition}%

\theoremstyle{definition}%
\newtheorem{definition}[theorem]{Definition}%
\newtheorem{assumption}[theorem]{Assumption}%

\theoremstyle{remark}%
\newtheorem*{remark}{Remark}%

\numberwithin{equation}{section}%
\numberwithin{theorem}{section}%

\title[Scalar conservation laws with Markov initial conditions]%
{Scalar conservation laws with monotone pure-jump Markov initial
  conditions}

\author[Kaspar]{David C. Kaspar}%
\address{%
  Division of Applied Mathematics \\
  Box F \\
  Brown University \\
  Providence, RI 02912}%
\email{david\_kaspar@brown.edu}%
\thanks{DK was supported by NSF DMS-1106526 and the Robinson
  Fellowship as the research reported herein was conducted, and NSF
  DMS-1148284 as this article was written.}
  
\author[Rezakhanlou]{Fraydoun Rezakhanlou}%
\address{%
  Department of Mathematics \\
  University of California \\
  Berkeley, CA 94720-3840}%
\email{rezakhan@math.berkeley.edu}%
\thanks{FR is supported by NSF grant DMS-1407723}

\date{\today}


\begin{document}

\maketitle

\begin{abstract}
  In 2010 Menon and Srinivasan published a conjecture for the
  statistical structure of solutions $\rho$ to scalar conservation
  laws with certain Markov initial conditions, proposing a kinetic
  equation that should suffice to describe $\rho(x,t)$ as a stochastic
  process in $x$ with $t$ fixed.  In this article we verify an
  analogue of the conjecture for initial conditions which are bounded,
  monotone, and piecewise constant.  Our argument uses a particle
  system representation of $\rho(x,t)$ over $0 \leq x \leq L$ for $L >
  0$, with a suitable random boundary condition at $x = L$.
\end{abstract}



\section{Introduction}
\label{sec:intro}

We are interested in the statistical properties of solutions to the
Cauchy problem for the scalar conservation law
\begin{equation}
  \label{eqn:scl}
  \left\{
    \begin{aligned}
      \rho_t &= H(\rho)_x && \text{in } \bbR \times (0,\infty) \\
      \rho &= \xi  && \text{on } \bbR \times \{t = 0\}
    \end{aligned}
  \right.
\end{equation}
where $\xi = \xi(x)$ is a stochastic process.  In the special case 
\begin{equation}
  \label{eqn:quadraticH}
  H(p) = - \frac{p^2}{2}
\end{equation}
\eqref{eqn:scl} is the well-known inviscid Burgers' equation, which
has often been considered with random initial data.  Burgers himself,
in his investigation of turbulence \cite{Burgers48}, was already
moving in this direction, though it would be some time before problems
in this area were addressed rigorously.

Several papers, including \cite{Carraro94, Carraro98, Chabanol04,
  Sinai92, Bertoin98, Bertoin00}, have developed a theory showing a
kind of integrability for the evolution of the \emph{law} of
$\rho(\cdot,t)$ for certain initial data.  This article is concerned
with pushing stochastic integrability results beyond the Burgers
setting to general scalar conservation laws, an effort begun in
earnest by Menon and Srinivasan \cite{Menon10,Menon12}.  Before
stating our main result, we first survey the major developments in
this subject.

This work is adapted from \cite{Kaspar14}.

\subsection{Background}
\label{sub:background}

The Burgers case $H(p) = - p^2/2$ has seen extensive interest.  Recall
that if
\begin{equation}
  u(x,t) = \int_{-\infty}^x \rho(x', t) \, dx'
\end{equation}
then $u(x,t)$ formally satisfies 
\begin{equation}
  \label{eqn:HJ}
  u_t + \frac{1}{2}(u_x)^2 = 0
\end{equation}
and is determined by the Hopf-Lax formula \cite{Hopf50}:
\begin{equation}
  \label{eqn:hopf-lax-burgers}
  u(x,t) = \inf_y \left\{\int_{-\infty}^y \xi(x') \, dx' +
    \frac{(x-y)^2}{2t}\right\}.
\end{equation}
The backward Lagrangian $y(x,t)$ is the rightmost minimizer in
\eqref{eqn:hopf-lax-burgers}, and it is standard \cite{Evans10} that
this determines both the viscosity solution $u(x,t)$ to \eqref{eqn:HJ}
and the entropy solution to $\rho_t + \rho \rho_x = 0$, 
\begin{equation}
  \label{eqn:rho-from-y}
  \rho(x,t) = \frac{x - y(x,t)}{t}.
\end{equation}

Groeneboom, in a paper \cite{Groeneboom89} concerned not with PDE but
concave majorants\footnote{Straightforward manipulations of
  \eqref{eqn:hopf-lax-burgers} relate this to a Legendre transform.}
and isotonic estimators, determined the statistics of $y(x,t)$ for the
case where $\xi(x)$ is white noise, which we understand to mean that
the initial condition for $u(x,t)$ in \eqref{eqn:HJ} is a standard
Brownian motion $B(x)$.  Among the results of \cite{Groeneboom89} is a
density for 
\begin{equation}
  \sup \{x \in \bbR : B(x) - cx^2 \text{ is maximal}\}
\end{equation}
expressed in terms of the Airy function.  It has been observed (see
e.g.\ \cite{Majumdar05}) that the Airy function seems to arise in a
surprising number of seemingly unrelated stochastic problems, and it
is notable that Burgers' equation falls in this class.

Several key papers appeared in the 1990s; we rely on the introduction
of Chabanol and Duchon's 2004 paper \cite{Chabanol04}, which recounts
some of this history, and discuss those portions most relevant for our
present purposes.  In 1992, Sinai \cite{Sinai92} explicitly connected
Burgers' equation with white noise initial data to convex minorants of
Brownian motion.  Three years later Avellaneda and E
\cite{AvellanedaE95} showed for the same initial data that the
solution $\rho(x,t)$ is a Markov process in $x$ for each fixed $t >
0$.

Carraro and Duchon's 1994 paper \cite{Carraro94} defined a notion of
\emph{statistical solution} to Burgers' equation.  This is a
time-indexed family of probability measures $\nu(t, \cdot)$ on a
function space, and the definition of solution is stated in terms of
the characteristic functional
\begin{equation}
  \label{eqn:char-func}
  \hat{\nu}(t,\phi) 
  = \int \exp \left(i \int \rho(x) \phi(x) \, dx \right) \nu(t,d\rho)
\end{equation}
for test functions $\phi$.  Namely, $\hat{\nu}(t,\phi)$ must satisfy
for each $\phi$ a differential equation obtained by formal
differentiation under the assumption that $\rho$ distributed according
to $\nu(t,d\rho)$ solves Burgers' equation.  This statistical solution
approach was further developed in 1998 by the same authors
\cite{Carraro98} and by Chabanol and Duchon \cite{Chabanol04}.  It
does have a drawback: given a (random) entropy solution $\rho(x,t)$ to
the inviscid Burgers' equation, the law of $\rho(\cdot, t)$ is a
statistical solution, but it is not clear that a statistical solution
yields a entropy solution, and at least one example is known
\cite{Bertoin98} when these notions differ.  Nonetheless,
\cite{Carraro94,Carraro98} realized that it was natural to consider
L\'evy process initial data, which set the stage for the next
development.

In 1998, Bertoin \cite{Bertoin98} proved a remarkable closure theorem
for L\'evy initial data.  We quote this, with adjustments to match our
notation.
\begin{theorem}[{\cite[Theorem 2]{Bertoin98}}]
  \label{thm:bertoin}
  Consider Burgers' equation with initial data $\xi(x)$ which is a
  L\'evy process without positive jumps for $x \geq 0$, and $\xi(x) =
  0$ for $x < 0$.  Assume that the expected value of $\xi(1)$ is
  positive, $\bbE \xi(1) \geq 0$.  Then, for each fixed $t > 0$, the
  backward Lagrangian $y(x,t)$ has the property that
  \begin{equation}
    y(x,t) - y(0,t)
  \end{equation}
  is independent of $y(0,t)$ and is in the parameter $x$ a
  subordinator, i.e.\ a nondecreasing L\'evy process.  Its distribution
  is the same as that of the \emph{first} passage process
  \begin{equation}
    x \mapsto \inf \{z \geq 0 : t \xi(z) + z > x\}.
  \end{equation}
  Further, denoting by $\psi(s)$ and $\Theta(t,s)$ $(s \geq 0)$ the
  Laplace exponents of $\xi(x)$ and $y(x,t) - y(0,t)$,
  \begin{align}
    \label{eqn:initial-exponent}
    \bbE \exp(s \xi(x)) &= \exp(x \psi(s)) \\
    \label{eqn:laplace-exponent}
    \bbE \exp[s(y(x,t) - y(0,t))] &= \exp(x \Theta(t,s)),
  \end{align}
  we have the functional identity
  \begin{equation}
    \label{eqn:laplaceidentity}
    \psi(t\Theta(t,s)) + \Theta(t,s) = s.
  \end{equation}
\end{theorem}

\begin{remark}
  The requirement $\bbE \xi(1) \geq 0$ can be relaxed, with minor
  modifications to the theorem, in light of the following elementary
  fact.  Suppose that $\xi^1(x)$ and $\xi^2(x)$ are two different
  initial conditions for Burgers' equation, which are related by
  $\xi^2(x) = \xi^1(x) + cx$.  It is easily checked using
  \eqref{eqn:hopf-lax-burgers} that the corresponding solutions
  $\rho^1(x,t)$ and $\rho^2(x,t)$ are related for $t > 0$ by
  \begin{equation}
    \label{eqn:rho-tilted}
    \rho^2(x,t) = \frac{1}{1+ct}
    \left[\rho^1\left(\frac{x}{1+ct},\frac{t}{1+ct}\right) + cx\right].
  \end{equation}
  This observation is found in a paper of Menon and Pego
  \cite{Menon07}, but (as this is elementary) it may have been known
  previously.  Using \eqref{eqn:rho-tilted} we can adjust a
  statistical description for a case where $\bbE \xi(1) \geq 0$ to
  cover the case of a L\'evy process with general mean drift.
\end{remark}

We find Theorem~\ref{thm:bertoin} remarkable for several reasons.
First, in light of \eqref{eqn:rho-from-y}, it follows immediately that
the solution $\rho(x,t) - \rho(0,t)$ is for each fixed $t$ a L\'evy
process in the parameter $x$, and we have an example of an
infinite-dimensional, nonlinear dynamical system (the PDE, Burgers'
equation) which preserves the independence and homogeneity properties
of its random initial configuration.  Second, the distributional
characterization of $y(x,t)$ is that of a \emph{first} passage
process, where the definition of $y(x,t)$ following
\eqref{eqn:hopf-lax-burgers} is that of a \emph{last} passage process.
Third, \eqref{eqn:laplaceidentity} can be used to show
\cite{MenonMiracle12} that if $\psi(t,q)$ is the Laplace exponent of
$\rho(x,t) - \rho(0,t)$, then
\begin{equation}
  \label{eqn:laplaceburgers}
  \psi_t + \psi \psi_q = 0
\end{equation}
for $t > 0$ and $q \in \bbC$ with nonnegative real part.  This shows
for entropy solutions what had previously been observed by Carraro and
Duchon for statistical solutions \cite{Carraro98}, namely that the
Laplace exponent \eqref{eqn:laplace-exponent} evolves according to
Burgers' equation!

In 2007 Menon and Pego \cite{Menon07} used the L\'evy-Khintchine
representation for the Laplace exponent \eqref{eqn:laplace-exponent}
and observed that the evolution according to Burgers' equation in
\eqref{eqn:laplaceburgers} corresponds to a Smoluchowski coagulation
equation \cite{Smoluchowski16, Aldous99}, with additive collision
kernel, for the jump measure of the L\'evy process $y(\cdot, t)$.  The
jumps of $y(\cdot, t)$ correspond to shocks in the solution
$\rho(\cdot, t)$.  The relative velocity of successive shocks can be
written as a sum of two functions, one depending on the positions of
the shocks and the other proportional to the sum of the sizes of the
jumps in $y(\cdot, t)$.  Regarding the sizes of the jumps as the usual
masses in the Smoluchowski equation, it is plausible that Smoluchowski
equation with additive kernel should be relevant, and \cite{Menon07}
provides the details that verify this.

It is natural to wonder whether this evolution through Markov
processes with simple statistical descriptions is a \emph{miracle}
\cite{MenonMiracle12} confined to the Burgers-L\'evy case, or an
instance of a more general phenomenon.  However, extending the results
of Bertoin \cite{Bertoin98} beyond the Burgers case $H(p) =
-\frac{1}{2} p^2$ remains a challenge.  A different particular case,
corresponding to $L(q) = (-H)^*(q) = |q|$, is a problem of determining
Lipschitz minorants, and has been investigated by Abramson and Evans
\cite{AbramsonEvans13}.  From the PDE perspective this is not as
natural, since $(-H)^*(q) = |q|$ corresponds to
\begin{equation}
  -H(p) = +\infty \, \bOne(|p| > 1),
\end{equation}
i.e.\ $-H(p)$ takes the value $0$ on $[-1,+1]$ and is equal to
$+\infty$ elsewhere.  So \cite{AbramsonEvans13}, while very
interesting from a stochastic processes perspective, has a specialized
structure which is rather different from those cases we will consider.

The biggest step toward understanding the problem for a wide class of
$H$ is found in a 2010 paper of Menon and Srinivasan \cite{Menon10}.
Here it is shown that when the initial condition $\xi$ is a
\emph{spectrally negative}\footnote{i.e.\ without positive jumps}
strong Markov process, the backward Lagrangian process $y(\cdot, t)$
and the solution $\rho(\cdot, t)$ remain Markov for fixed $t > 0$, the
latter again being spectrally negative.  The argument is adapted from
that of \cite{Bertoin98} and both \cite{Bertoin98, Menon10} use the
notion of splitting times (due to Getoor \cite{Getoor79}) to verify
the Markov property according to its bare definition.  In the
Burgers-L\'evy case, the independence and homogeneity of the
increments can be shown to survive, from which additional regularity
is immediate using standard results about L\'evy processes
\cite{Kallenberg02}.  As \cite{Menon10} points out, without these
properties it is not clear whether a Feller process initial condition
leads to a Feller process in $x$ at later times.  Nonetheless,
\cite{Menon10} presents a very interesting conjecture for the
evolution of the generator of $\rho(\cdot, t)$, which has a remarkably
nice form and follows from multiple (nonrigorous, but persuasive)
calculations.

We now give a partial statement of this conjecture.  The generator
$\cA$ of a stationary, spectrally negative Feller process acts on test
functions $J \in C^\infty_c(\bbR)$ by
\begin{equation}
  \label{eqn:feller}
  (\cA J)(y) = b(y) J'(y) + \int_{-\infty}^y (J(z) - J(y))
  \, f(y, dz)
\end{equation}
where $b(y)$ characterizes the drift and $f(y, \cdot)$ describes the
law of the jumps.  If we allow $b$ and $f$ to depend on $t$, we have a
family of generators.  The conjecture of \cite{Menon10} is that the
evolution of the generator $\cA$ for $\rho(\cdot, t)$ is given by the
Lax equation
\begin{equation}
  \label{eqn:lax}
  \dot{\cA} = [\cA, \cB] = \cA \cB - \cB \cA
\end{equation}
for $\cB$ which acts on test functions $J$ by
\begin{equation}
  \label{eqn:timegen}
  (\cB J)(y) = - H'(y) b(t,y) J'(y) - 
  \int_{-\infty}^y \frac{H(y) - H(z)}{y-z} (J(z) - J(y)) \, f(t, y,
  dz).
\end{equation}
An equivalent form of the conjecture \eqref{eqn:lax} involves a
kinetic equation for $f$.  The key result in the present article
verifies that this kinetic equation holds in the special case we will
consider.  Before we state this, let us establish our working
notation.

\subsection{Notation}
\label{sub:notation}

Here we collect some of the various notation used later in the
article.

Write $H[p_1,\ldots,p_n]$ for the $n^{\mathrm{th}}$ divided difference
of $H$ through $p_1,\ldots,p_n$.  For each $n$ this function is
symmetric in its arguments, and given by the following formulas in the
cases $n = 2$ and $n = 3$ where $p_1,p_2,p_3$ are distinct:
\begin{equation}
  H[p_1,p_2] = \frac{H(p_2) - H(p_1)}{p_2 - p_1} \quad \text{and}
  \quad H[p_1,p_2,p_3] = \frac{H[p_2,p_3] - H[p_1,p_2]}{p_3 - p_1}.
\end{equation}
The definition for general $n$ and standard properties can be found in
several numerical analysis texts, including \cite{Hildebrand87}.

We write $\delta_x$ for the usual point mass assigning unit mass to
$x$ and zero elsewhere.

For $C > 0$ and $n$ a positive integer, write 
\begin{equation}
  \Delta^C_n = \{(a_1,\ldots,a_n) \in \bbR^n : 0 < a_1 < \cdots < a_n
  < C\}
\end{equation}
and $\overline{\Delta^C_n}$ for the closure of this set in $\bbR^n$.
We write $\partial \Delta^C_n$ for the boundary of the simplex, and
$\partial_i \Delta^C_n$ for its various faces: 
\begin{equation}
  \begin{aligned}
    \partial_0 \Delta^C_n &= \{(a_1,\ldots,a_n) \in \partial \Delta^C_n
    : a_1 = 0\} \\
    \partial_i \Delta^C_n &= \{(a_1,\ldots,a_n) \in \partial
    \Delta^C_n : a_i = a_{i+1}\}, \qquad i = 1,\ldots,n-1 \\
    \partial_n \Delta^C_n &= \{(a_1,\ldots,a_n) \in \partial
    \Delta^C_n : a_n = C\}.
  \end{aligned}
\end{equation}

Write $\cM$ for the set of finite, regular (signed) measures on
$[0,P]$, which is a Banach space when equipped with the total
variation norm $\|\cdot\|_{\TV}$.  Call its nonnegative subset
$\cM_+$.

Let $\cK$ denote the set of bounded signed kernels from $[0,P]$ to
$[0,P]$, i.e.\ the set of $k : [0,P] \mapsto \cM$ which are measurable
when $[0,P]$ is endowed with its Borel $\sigma$-algebra and $\cM$ is
endowed with the $\sigma$-algebra generated by evaluation on Borel
subsets of $[0,P]$, and which satisfy 
\begin{equation}
  \|k\| = \sup \{ \|k(\rho_-, \cdot)\|_{\TV} : \rho_- \in [0,P]\}
  < \infty.
\end{equation}
Observe that $\cK$ is a Banach space; completeness holds since a
Cauchy sequence $k_n$ has $k_n(\rho_-,\cdot)$ Cauchy in total
variation for each $\rho_-$, and we obtain a pointwise limit
$k(\rho_-,\cdot)$.  Measurability of $k(\cdot, A)$ for each Borel
$A \subseteq [0,P]$ then holds since this is a real-valued pointwise
limit of $k_n(\cdot, A)$.  Write $\cK_+$ for the subset of $\cK$ with
range contained in $\cM_+$.

\subsection{Main result}
\label{sub:main}

In this section we provide a statistical description of solutions to
the scalar conservation law when the initial condition is an
increasing, pure-jump Markov process given by a rate kernel $g$.  For
this we require some assumptions on the rate kernel $g$ and the
Hamiltonian $H$.

\begin{assumption}
  \label{asm:const-rate}
  The initial condition $\xi = \xi(x)$ is a pure-jump Markov process
  starting at $\xi(0) = 0$ and evolving for $x > 0$ according to a
  rate kernel $g(\rho_-,d\rho_+)$.  We assume that for some
  constant $P > 0$ the kernel $g$ is supported on
  \begin{equation}
    \{(\rho_-,\rho_+) : 0 \leq \rho_- \leq \rho_+ \leq P\}
  \end{equation}
  and has total rate which is constant in $\rho_-$: 
  \begin{equation}
    \lambda = \int g(\rho_-, d\rho_+)
  \end{equation}
  for all $0 \leq \rho_- \leq P$.  In particular, $g \in \cK_+$.
\end{assumption}

\begin{assumption}
  \label{asm:H}
  The Hamiltonian function $H : [0,P] \to \bbR$ is smooth, convex, has
  nonnegative right-derivative at $p = 0$ and noninfinite
  left-derivative at $p = P$.
\end{assumption}

We provide for $x \geq 0$ a statistical description consisting of a
one-dimensional marginal at $x = 0$ and a rate kernel generating the
rest of the path.  The evolution of the rate kernel is given by the
following kinetic equation, and the evolution of the marginal will be
described in terms of the solution to the kinetic equation.

\begin{definition}
  \label{def:kinetic-solution}
  We say that a continuous mapping $f : [0,\infty) \to \cK_+$ is a
  \emph{solution of the kinetic equation}
  \begin{equation}
    \left\{
      \begin{aligned}
        f_t &= \cL^\kappa f \\
        f(0,\rho_-,d\rho_+) &= g(\rho_-,d\rho_+),
      \end{aligned}
    \right.
  \end{equation}
  where 
  \begin{equation}
    \label{eqn:kinetic}
    \begin{aligned}
      &\cL^{\kappa} f(t,\rho_-,d\rho_+) \\
      &= \int (H[\rho_*,\rho_+] - H[\rho_-,\rho_*])
      f(t,\rho_-,d\rho_*) f(t,\rho_*,d\rho_+) \\
      &\quad - \left[\int H[\rho_+,\rho_*] f(t,\rho_+,d\rho_*) - \int
        H[\rho_-,\rho_*] f(t,\rho_-,d\rho_*)\right]
      f(t,\rho_-,d\rho_+),
    \end{aligned}
  \end{equation}
  provided that
  $\theta \mapsto \cL^{\kappa} f(\theta, \cdot, \cdot) \in \cK$ is
  Bochner-integrable and
  \begin{equation}
    f(t,\cdot,\cdot) - f(s,\cdot,\cdot) = \int_s^t \cL^\kappa
    f(\theta, \cdot, \cdot) \, d\theta
  \end{equation}
  for all $0 \leq s \leq t < \infty$.
\end{definition}

\begin{definition}
  \label{def:marginal-solution}
  We say that a continuous mapping $\ell : [0,\infty) \to \cM_+$ is a
  \emph{solution of the marginal equation} 
  \begin{equation}
    \left\{
      \begin{aligned}
        \ell_t &= \cL^0 \ell \\
        \ell(0,d\rho_0) &= \delta_0(d\rho_0),
      \end{aligned}
    \right.
  \end{equation}
  where 
  \begin{multline}
    \label{eqn:marginal}
    \cL^0 \ell(t,d\rho_0) = \int H[\rho_*,\rho_0] \ell(t,d\rho_*)
    f(t,\rho_*,d\rho_0) \\
    - \left[\int H[\rho_0,\rho_*] f(t, \rho_0, d\rho_*)\right]
    \ell(t,d\rho_0),
  \end{multline}
  provided that $\theta \mapsto \cL^0 \ell(\theta, \cdot) \in
  \cM$ is Bochner-integrable and 
  \begin{equation}
    \ell(t,\cdot) - \ell(s,\cdot) = \int_s^t \cL^0 \ell(\theta, \cdot)
    \, d\theta
  \end{equation}
  for all $0 \leq s \leq t < \infty$.
\end{definition}

\begin{theorem}
  \label{thm:kinetic}
  Suppose Assumptions~\ref{asm:const-rate} and \ref{asm:H} hold.  Then
  the kinetic and marginal equations have unique solutions, in the
  sense of Definitions~\ref{def:kinetic-solution} and
  \ref{def:marginal-solution}.  Furthermore, the total integrals are
  conserved:
  \begin{align}
    \label{eqn:total-int}
    \lambda &= \int f(t,\rho_-,d\rho_+) \\
    1 &= \int \ell(t,d\rho_0) 
  \end{align}
  for all $t \geq 0$ and all $0 \leq \rho_- \leq P$.  Finally, if 
  \begin{equation}
    g(\rho_-, [0,\rho_-] \cup \{P\}) = 0
  \end{equation}
  for all $0 \leq \rho_- < P$, then $f(t,\cdot,\cdot)$ has the same
  property for $t > 0$.
\end{theorem}

The kernels described by the Theorem~\ref{thm:kinetic} are precisely
what we need to describe the statistics of the solution $\rho$, which
brings us to our main result:

\begin{theorem}
  \label{thm:unbounded-sys}
  When Assumptions~\ref{asm:const-rate} and \ref{asm:H} hold, the
  solution $\rho$ to
  \begin{equation}
    \label{eqn:scl-unbounded}
    \left\{
      \begin{aligned}
        \rho_t &= H(\rho)_x && (x,t) \in \bbR \times (0,\infty) \\
        \rho &= \xi && (x,t) \in \bbR \times \{t = 0\}
      \end{aligned}
    \right.
  \end{equation}
  for each fixed $t > 0$ has $x = 0$ marginal given by
  $\ell(t,d\rho_0)$ and for $0 < x < \infty$ evolves according to rate
  kernel $f(t,\rho_-,d\rho_+)$.
\end{theorem}

\begin{remark}
  Theorems \ref{thm:kinetic} and \ref{thm:unbounded-sys} establish
  rigorously the conjectured \cite{Menon10} evolution according to the
  Lax pair \eqref{eqn:lax}, within the present hypotheses.  See
  \cite[Section 2.7]{Menon10} for a calculation showing the
  equivalence of the kinetic and Lax pair formulations, which
  simplifies considerably in the present case due to the absence of
  drift terms.  The Lax pair and integrable systems approach (in the
  case of finitely many states, where the generator is a triangular
  matrix) have been further explored by Menon \cite{Menon12} and in a
  forthcoming work by Li \cite{Li15}.
\end{remark}

\subsection{Organization}

The remainder of this article is organized as follows.  In
Section~\ref{sec:particle} we show that
Theorem~\ref{thm:unbounded-sys} will follow from a similar statistical
characterization for the solution to the PDE over $x \in [0,L]$ with a
time-dependent \emph{random} boundary condition at $x = L$.  The
latter we can study using a sticky particle system whose dimension is
random and unbounded, but almost surely finite.  Elementary arguments
are used to check that our candidate for the law matches the evolution
of the random initial condition according to the dynamics.  Next in
Section~\ref{sec:kinetic} we show existence and uniqueness of the
solutions to marginal and kinetic equations of
Theorem~\ref{thm:kinetic}, which are needed to construct the candidate
law.  The concluding Section~\ref{sec:conclusion} indicates some
desired extensions and similar questions for future work.  To keep the
main development concise, proofs of lemmas have been deferred to
Appendix~\ref{app:proofs}.



\section{A random particle system}
\label{sec:particle}

As functions of $x$, the solutions $\rho(x,t)$ we consider all have
the form depicted in Figure~\ref{fig:monotone}.  From a PDE
perspective this situation is standard---a concatenation of Riemann
problems for the scalar conservation law---and we can describe the
solution completely in terms of a particle system.  Each shock
consists of some number of particles stuck together, and the particles
move at constant velocities according to the Rankine-Hugoniot
condition except when they collide.  The collisions are totally
inelastic.

\begin{figure}[h]
  \begin{center}
    \includegraphics{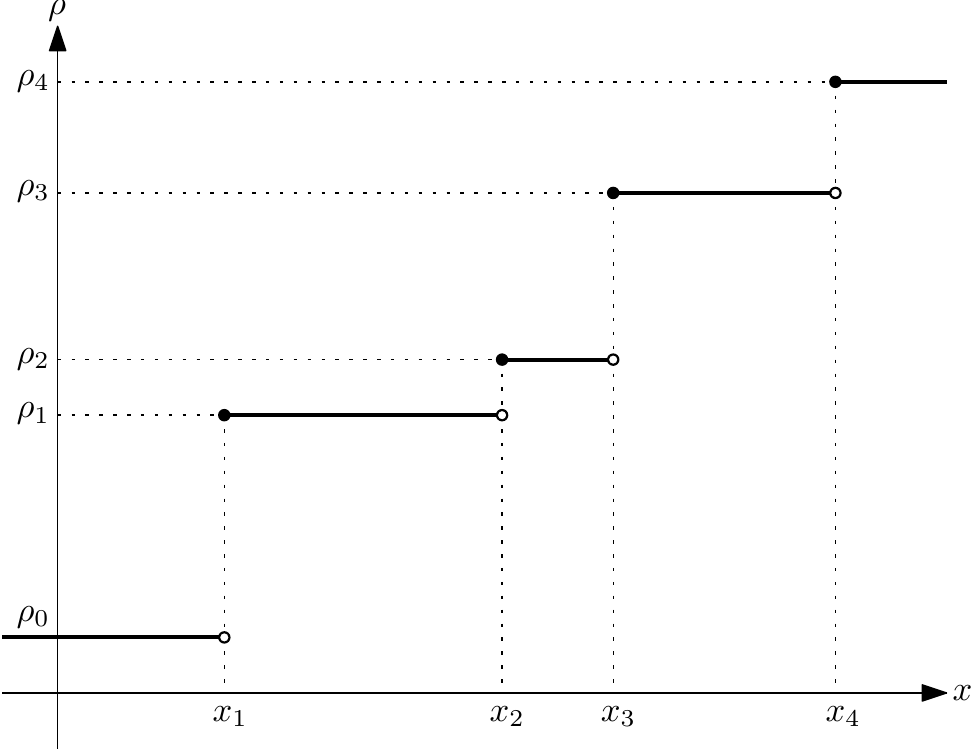}
    \caption{For each $t > 0$, the solution $\rho(x,t)$ is a
      nondecreasing, pure-jump process in $x$.  We will see that for
      any fixed $L > 0$, we have a.s.\ finitely many jumps for $x \in
      [0,L]$ and that $\rho(\cdot,t)$ on this interval can be
      described by two (finite) nondecreasing sequences
      $(x_1,\ldots,x_N; \rho_0,\ldots,\rho_N)$.}
    \label{fig:monotone}
  \end{center}
\end{figure}

The utility of this particle description is lessened, however, by the
fact that the dynamics are quite infinite dimensional for our
pure-jump initial condition $\xi(x)$, $x \in [0,+\infty)$.  To have a
simple description as motion at constant velocities, punctuated by
\emph{occasional} collisions, we might argue that on each fixed
bounded interval of space we have finitely many collisions in a
bounded interval of time, and piece together whatever statistical
descriptions we might obtain for the solution on these various
intervals.  Inspired by those situations in statistical mechanics
where boundary conditions become irrelevant in an infinite-volume
limit, we pursue a different approach.  We construct solutions to a
problem on a bounded space interval $[0,L]$, and choose a random
boundary condition at $x = L$ to obtain an exact match with the
kinetic equations.  The involved analysis will all pertain to the
following result.

\begin{theorem}
  \label{thm:bounded-sys}
  Suppose Assumptions~\ref{asm:const-rate} and \ref{asm:H} hold.  For
  any fixed $L > 0$, consider the scalar conservation law
  \begin{equation}
    \label{eqn:scl-bounded}
    \left\{
      \begin{aligned}
        \rho_t &= H(\rho)_x && (x,t) \in (0,L) \times (0,\infty) \\
        \rho &= \xi && x \in [0,L] \times \{t = 0\} \\
        \rho &= \zeta && (x,t) \in \{x = L\} \times (0,\infty)
      \end{aligned}
    \right.
  \end{equation}
  with initial condition $\xi$ (restricted to $[0,L]$), open
  boundary\footnote{Using Assumption~\ref{asm:H} and $\xi,\zeta \geq
    0$, the shocks and characteristics only flow outward across $x =
    0$.  Any boundary condition we would assign, unless it involved
    negative values, would thus be irrelevant.} at $x = 0$, and random
  boundary $\zeta$ at $x = L$.  Suppose the process $\zeta$ has
  $\zeta(0) = \xi(L)$ and evolves according to the time-dependent rate
  kernel $H[\rho_,\rho_+] f(t,\rho_,d\rho_+)$ independently of $\xi$
  given $\xi(L)$.  Then for all $t > 0$ the law of $\rho(\cdot, t)$ is
  as follows:
  \begin{itemize}
  \item[(i)] the $x = 0$ marginal is $\ell(t,d\rho_0)$, and
  \item[(ii)] the rest of the path is a pure-jump process with rate
    kernel $f(t,\rho_-,d\rho_+)$.
  \end{itemize}
\end{theorem}

To prove our main result we can send $L \to \infty$, applying
Theorem~\ref{thm:bounded-sys} on each $[0,L]$, and use bounded speed
of propagation to limit the respective influences of far away
particles (unbounded system) or truncation with random boundary
(bounded system).  The argument is quite short.

\begin{proof}[Proof of Theorem~\ref{thm:unbounded-sys}]
  Fix any $t > 0$.  We will write $\hat{\rho}(x,t)$ for the solution
  to \eqref{eqn:scl-unbounded} over the semi-infinite $x$-interval
  $[0,+\infty)$ with initial data $\hat{\xi}(x)$.  We take the
  right-continuous version of the solution.  For $L$ to be specified
  shortly, write $\rho(x,t)$ for the solution and $\xi(x)$ for the
  initial data corresponding to \eqref{eqn:scl-bounded}.

  Fix any $x_1,\ldots,x_k \in [0,+\infty)$, and let $X =
  \max\{x_1,\ldots,x_k\}$.  Choose $L > X + t H'(P)$.  We couple the
  bounded and unbounded systems by requiring 
  \begin{equation}
    \hat{\xi}|_{[0,L]} = \xi,
  \end{equation}
  allowing the random boundary $\zeta$ to evolve independently of
  $\hat{\xi}$ given $\hat{\xi}(L)$

  Recall \cite{Evans10} that the scalar conservation law has finite
  speed of propagation.  Our solutions are bounded in $[0,P]$, so the
  speed is bounded by $H'(P)$.  Since $\hat{\rho}(x,0)$ and
  $\rho(x,0)$ are a.s.\ equal on $[0,L]$, we have also
  $\hat{\rho}(x,t) = \rho(x,t)$ a.s.\ for $x \in [0,L-tH'(P)] \supset
  [0,X]$.  From this we deduce the distributional equality
  \begin{equation}
    (\hat{\rho}(x_1,t),\ldots,\hat{\rho}(x_k,t)) \stackrel{d}{=}
    (\rho(x_1,t),\ldots,\rho(x_k,t)).
  \end{equation}
  By Theorem~\ref{thm:bounded-sys}, the latter distribution is exactly
  that of a process started according to $\ell(t,d\rho_0)$, evolving
  according to rate kernel $f(t,\rho_-,d\rho_+)$.  This process is
  terminated deterministically at $x = L$, which does not alter
  finite-dimensional distributions prior to $x = L$.  Since
  $\rho(x,t)$ is right-continuous and has the correct
  finite-dimensional distributions, the result follows.
\end{proof}

\subsection{The dynamics}
\label{sub:dynamics}

Our work in the remainder of this section is to prove
Theorem~\ref{thm:bounded-sys}.  We begin by describing precisely those
particle dynamics which determine the solution to the PDE.

Figure~\ref{fig:monotone} illustrates a parametrization of a
nondecreasing pure-jump process on $[0,+\infty)$ with heights
$\rho_0,\rho_1,\rho_2,\ldots$ and jump locations $x_1,x_2,\ldots$.
Our sign restriction on the jumps excludes rarefaction waves, and we
have only constant values separated by shocks.  Going forward in
time, the shocks move according to the Rankine-Hugoniot condition
\begin{equation}
  \dot{x}_i = -H[\rho_{i-1},\rho_i]
\end{equation}
until they collide.  We say that each shock consists initially of one
\emph{particle} moving at the velocity indicated above.  Then result
of a collision can be characterized in two equivalent ways:
\begin{itemize}
\item[(i)] At the first instant when $x_i = x_{i+1}$, the
  $i^{\mathrm{th}}$ particle is annihilated, and the velocity of the
  $(i+1)^{\mathrm{th}}$ particle changes from $-H[\rho_i,\rho_{i+1}]$
  to
  \begin{equation}
    \dot{x}_{i+1} = - H[\rho_{i-1},\rho_{i+1}].
  \end{equation}
  In the case\footnote{which will almost surely not occur for the
    randomness we consider} where several consecutive particles
  collide with each other at the same instant, all but the rightmost
  particle is annihilated.  Since we only seek a statistical
  description for $x \geq 0$, we annihilate the first particle when
  $x_1 = 0$, replace $\rho_0$ with $\rho_1$, and relabel the other
  particles accordingly.
\item[(ii)] When $x_{i-1} < x_i = x_{i+1} = \cdots = x_{j-1} = x_j <
  x_{j+1}$, the particles $i$ through $j$ all move with common
  velocity $-H[\rho_{i-1},\rho_j]$.  Following Brenier et al
  \cite{Brenier98,Brenier13} we call these a \emph{sticky particle
    dynamics}.  We can additionally take the position $x = 0$ to be
  absorbing: $\dot{x_i} = 0$ whenever $x_i = 0$.
\end{itemize}
We adopt the first viewpoint for convenience, as this is compatible
with \cite{Davis84} (see the text following
Definition~\ref{def:dynamics} below), but a suitable argument could be
given for the second alternative as well.

\begin{definition}
  \label{def:config-space}
  For $L$ as in Theorem~\ref{thm:bounded-sys}, the configuration space
  $Q$ for the sticky particle dynamics is
  \begin{equation}
    Q = \bigsqcup_{n=0}^\infty Q_n, \qquad Q_n = \Delta^L_n
    \times \overline{\Delta^P_{n+1}}.
  \end{equation}
  A typical configuration is $q =
  (x_1,\ldots,x_n;\rho_0,\ldots,\rho_n) \in Q_n$ when $n>0$, or $q =
  (\rho_0) \in Q_0=\{\rho_0: \ 0\le \rho_0\le P\}$ when $n=0$.
\end{definition}

\begin{definition}
  \label{def:dynamics}
  Our notation for the particle dynamics is as follows:
  \begin{itemize}
  \item[(i)] For $0 \leq s \leq t$ and $q \in Q$, write 
    \begin{equation}
      \phi_s^t q = \phi_0^{t-s} q
    \end{equation}
    for the deterministic evolution from time $s$ to $t$ of the
    configuration $q$ according to the annihilating particle dynamics
    for the PDE, \emph{without} random entry dynamics at $x = L$.
  \item[(ii)] Given a configuration $q =
    (x_1,\ldots,x_n,\rho_0,\ldots,\rho_n)$ and $\rho_+ > \rho_n$,
    write $\epsilon_{\rho_+} q$ for the configuration
    $(x_1,\ldots,x_n,L,\rho_0,\ldots,\rho_n,\rho_+)$.
  \item[(iii)] Write $\Phi_s^t q$ for the \emph{random} evolution of
    the configuration according to deterministic particle dynamics
    interrupted with random entries at $x = L$ according to the
    boundary process $\zeta$ of \eqref{eqn:scl-bounded}, where the
    latter has been started at time $s$ with value $\rho_n$.  In
    particular, if the jumps of $\zeta$ between times $s$ and $t$
    occur at times $s < \tau_1 < \cdots < \tau_k < t$ with values
    $\rho_{n+1} < \cdots < \rho_{n+k}$, then
    \begin{equation}
      \Phi_s^t q = \phi_{\tau_k}^t \epsilon_{\rho_{n+k}}
      \phi_{\tau_{k-1}}^{\tau_k} \epsilon_{\rho_{n+k-1}} \cdots
      \phi_{\tau_1}^{\tau_2} \epsilon_{\rho_{n+1}} \phi_s^{\tau_1} q.
    \end{equation}
  \end{itemize}
\end{definition}

\begin{proposition}
  For any $s,q$, the process $\Phi_s^t q$ is strong Markov.
\end{proposition}

This assertion follows after recognizing $\Phi_s^t q$ as a
piecewise-deterministic Markov process described in some generality by
Davis \cite{Davis84}.  Namely, we augment the configuration space $Q$
to add the time parameter to each component $Q_n$, and then have a
deterministic flow according to the vector field
\begin{equation}
  (1; -H[\rho_0,\rho_1], \ldots, -H[\rho_{n-1},\rho_n]; 0, \ldots, 0).
\end{equation}
With rate $\int H[\rho_n,\rho_+] f(t,\rho_n,d\rho_+)$ we jump to the
indicated point in $Q_{n+1}$, and upon hitting a boundary we
transition to $Q_k$ for suitable $k < n$, annihilating particles in
the manner described above.

\subsection{Checking the candidate measure}
\label{sub:candidate}

Our goal is to take $q$ distributed according to the initial condition
and exactly describe the law of $\Phi_0^T q$ for each $T > 0$.  Using
the kinetic~\eqref{eqn:kinetic} and marginal
equations~\eqref{eqn:marginal}, we construct for each time $t \geq 0$
a candidate law $\mu(t,dq)$ on $Q$ as follows.  Take $N$ to be Poisson
with rate $\lambda L$, $x_1,\ldots,x_N$ uniform on $\Delta^L_N$, and
$\rho_0,\ldots,\rho_N$ distributed on $\overline{\Delta^P_{N+1}}$
according to the marginal $\ell$ and transitions $f$ independently of
the $x_i$:
\begin{equation}
  \label{eqn:candidate}
  e^{-\lambda L} \sum_{n=0}^\infty \delta_n(dN)\mu_n(t,dq),
\end{equation}
where $\mu_0(t,dq) = \ell(t,d\rho_0)$ and
\begin{equation}
  \mu_n(t,dq) =  \bOne_{\Delta^L_n}(x_1,\ldots,x_n) \, dx_1 \cdots dx_n \,
  \ell(t,d\rho_0) \prod_{j=1}^n f(t,\rho_{j-1},d\rho_j),
\end{equation}
for $n>0$. When we have verified \eqref{eqn:total-int}, it will be
immediate that the total mass of this is one.

We decompose the mapping from configurations $q$ to solutions (as
functions of $x$ over $[0,L]$) into the map from $q$ to the measure
\begin{equation}
  \label{eqn:rho-measure}
  \pi(q, dx) = \rho_0 \delta_0 + \sum_{i=1}^n (\rho_i - \rho_{i-1})
  \delta_{x_i} \qquad ( q \in Q_n)
\end{equation}
and integration over $[0,x]$.  We claim that when $q$ is distributed
according to $\mu(0,dq)$, the law of the measure $\pi(\Phi_0^t q,
\cdot)$ is identical to that of $\pi(q', \cdot)$ where $q'$ is
distributed according to $\mu(t,dq')$.

We now describe the structure of the proof of
Theorem~\ref{thm:bounded-sys}.  Fix some time $T > 0$ and consider
$F(t,q) = \bbE G(\Phi_t^T q)$ where $G$ takes the form of a Laplace
functional:
\begin{equation}
  \label{eqn:laplacefunctional}
  G(q) = \exp\left(-\int J(x) \, \pi(q,dx) \right) = \exp\left(-\rho_0
    J(0) - \sum_{i=1}^n (\rho_i - \rho_{i-1}) J(x_i)\right)
\end{equation}
for $J \geq 0$ a continuous function on $[0,L]$.  We aim to
show that
\begin{equation}
  \label{eqn:unchanging}
  \frac{d}{dt} \int \bbE G(\Phi_t^T q) \mu(t,dq) = 0
\end{equation}
for $0 < t < T$, from which it will follow that
\begin{equation}
  \label{eqn:GzeroT}
  \int \bbE G(\Phi_0^T q) \, \mu(0,dq) = \int G(q) \, \mu(T,dq)
\end{equation}
for all $G$ of the form in \eqref{eqn:laplacefunctional}.  Using the
standard fact that Laplace functionals completely determine the law of
any random measure \cite{Kallenberg02}, this will suffice to show that
law of $\pi(\Phi_0^T q, dx)$ for $q$ distributed as $\mu(0, dq)$ is
precisely the pushforward through $\pi$ of $\mu(T, dq)$, and obtain
the result.

We might verify \eqref{eqn:unchanging} by establishing regularity for
$F(t,q) = \bbE G(\Phi_t^T q)$ in $t$ and the $x$-components of $q$.
We speculate that it should be possible to do so if $J$ is smooth with
$J'(0) = 0$ and we content ourselves to divide the configuration space
into finitely many regions, each of which corresponds to a definite
order of deterministic collisions in $\phi$.  On the other hand, the
measures $\mu(t, dq)$ enjoy considerable regularity (uniformity, in
fact) in $x$, and we pursue instead an argument along these lines.
Here continuity of $F(t,q)$ in $t$ will suffice.

\begin{lemma}
  \label{lem:tcontinuity}
  Let $G$ take the form of \eqref{eqn:laplacefunctional}.  Then
  $F(t,q) = \bbE G(\Phi_t^T q)$ is a bounded function which is
  uniformly continuous in $t$ uniformly in $q$:
  \begin{equation}
    \label{eqn:cont-modulus}
    w(\delta) = \sup \{ |F(t,q) - F(s,q)| \, : \, s,t \in [0,T], \,
    |t-s| < \delta, \, q \in Q\} \to 0
  \end{equation}
  as $\delta \to 0+$.
\end{lemma}

To differentiate in \eqref{eqn:unchanging} we will need to compare
$\bbE G(\Phi_s^T q)$ and $\bbE G(\Phi_t^T q)$ for $0 < s < t < T$.
Our next observation is that when the $t-s$ is small, we can separate
the deterministic and stochastic portions of the dynamics over the
time interval $[s,t]$.  The idea is that in an interval of time
$[s,t]$, the probability of multiple particle entries is $O((t-s)^2)$,
and a single particle entry has probability $O(t-s)$ which permits
additional $o(1)$ errors.

\begin{lemma}
  \label{lem:separate-random}
  Let $0 < s < t \leq T$ and $q \in Q_n$.  There exist a random
  variable $\tau \in (s,t)$ a.s., with law depending on $s,t,$ and $q$
  only through $\rho_n$, and a constant $C_{\lambda,H'(P)}$
  independent of $q,s,t$ so that
  \begin{multline}
    \left|F(s,q) - F(t,\phi_s^t q) - (t-s) \int [\bbE F(t,
      \epsilon_{\rho_+} \phi_s^\tau q) - F(t,q)] H[\rho_n,\rho_+]
      f(t,\rho_n,d\rho_+)\right| \\
    \leq C_{\lambda,H'(P)} [(t-s)^2 + (t-s) w(t-s)].
  \end{multline}
\end{lemma}

We proceed to an analysis of the deterministic portion of the flow,
$\phi_s^t q$, where $q$ is distributed according to $\mu(t,dq)$.  For
this we introduce some notation.  Given $q \in Q_n$, we write
\begin{equation}
  \sigma(q) = \min \{r \geq s \, : \phi_s^r q \in (\partial \Delta^L_n)
  \times \overline{\Delta^P_{n+1}}\}.
\end{equation}
For each fixed $n$, we consider several subsets of the set of
$n$-particle configurations $Q_n$.  In particular we need to separate
those configurations which experience an exit at $x = 0$ or a
collision in a time interval shorter than $t-s$.  For $i =
0,\ldots,n-1$, write: 
\begin{equation}
  \begin{aligned}
    A_i &= \{q \in Q_n : \sigma(q) \leq t \text{ and }
    \phi_s^{\sigma(q)} q \in (\partial_i \Delta^L_n) \times
    \overline{\Delta^P_{n+1}}\} \\
    U &= Q_n \setminus \bigcup_{i=0}^{n-1} A_i \\
    B &= \{q \in Q_n : x_n \geq L - (t-s) H[\rho_{n-1},\rho_n]\} \\
    V &= Q_n \setminus B.
  \end{aligned}
\end{equation}
Figure~\ref{fig:simplex} illustrates these sets in the case $n = 2$.
We observe in particular that $\phi_s^t U$ and $V$ differ by a
$\mu_n(t,\cdot)$-null set, and that the terms associated with the
boundary faces are related to configurations with one fewer particle.

\begin{lemma}
  \label{lem:boundary-approx}
  For each positive integer $n$ and with errors bounded by
  \begin{equation}
    C_{L,n,\lambda,H'(P)} [(t-s) + w(t-s)]
  \end{equation}
  we have the following approximations: 
  \begin{align}
    &(t-s)^{-1} \int F(t,\phi_s^t q) \bOne_{A_0}(q) \mu_n(t,dq)
    \notag \\
    &\quad \approx \int F(t,q) \frac{H[\rho_*,\rho_0]
      \ell(t,d\rho_*) f(t,\rho_*,d\rho_0)}{\ell(t,d\rho_0)}
    \mu_{n-1}(t,dq), \\
    &(t-s)^{-1} \int F(t,\phi_s^t q) \bOne_{A_i}(q) \mu_n(t,dq)
    \notag \\
    &\quad \approx \int F(t,q) \frac{(H[\rho_*,\rho_i] -
      H[\rho_{i-1},\rho_*]) f(t,\rho_{i-1},d\rho_*)
      f(t,\rho_*,d\rho_i)}{f(t,\rho_{i-1},d\rho_i)} \mu_{n-1}(t,dq)
    \\
    \intertext{for $i = 1, \ldots, n-1$, and}%
    &(t-s)^{-1} \int F(t,\phi_s^t q) \bOne_B(q) \mu_n(t,dq)
    \notag \\
    &\quad \approx \int F(t,\epsilon_{\rho_+} q)
    H[\rho_{n-1},\rho_+] f(t,\rho_{n-1},d\rho_+) \mu_{n-1}(t,dq).
  \end{align}
\end{lemma}

\begin{remark}
  It is \emph{essential} that the integral over $B$ can be
  approximated by an integral over
  $(x_1,\ldots,x_{n-1},L;\rho_0,\ldots,\rho_{n-1},\rho_+)$.  The
  measure $\cL^* \mu$ below does not have any singular factors like
  $\delta(x_n = L)$.  In particular, the result of integrating over
  $B$ will partially cancel with the term arising from random particle
  entries.
\end{remark}

\begin{figure}[h]
  \begin{center}
    \includegraphics[width=0.6\textwidth]{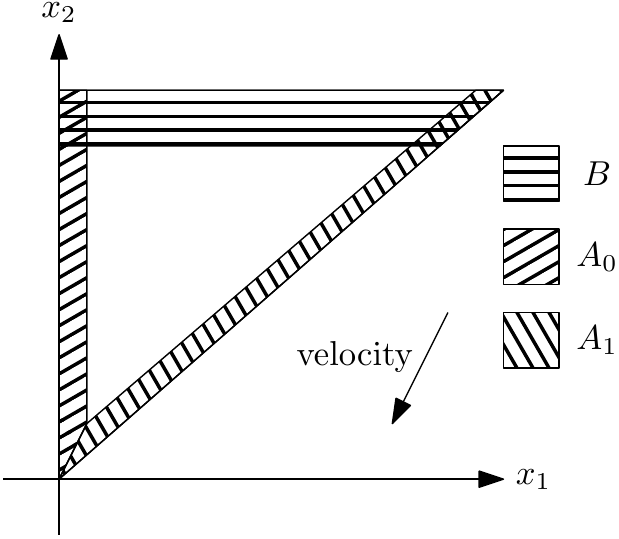}
    \caption{The deterministic flow $\phi$ on the $x$-simplex is
      translation at constant velocity unless this translation crosses
      a boundary of $\Delta^L_n$.  In the case $n = 2$ pictured above,
      points in $A_0$ and $A_1$ hit the boundary faces $x_1 = 0$ and
      $x_1 = x_2$, respectively.  The remaining portion of the simplex
      is mapped to the simplex minus the set $B$.}
    \label{fig:simplex}
  \end{center}
\end{figure}

The final ingredient for the proof of Theorem~\ref{thm:bounded-sys} is
the time derivative of $\mu(t,dq)$, which we now record.

\begin{lemma}
  \label{lem:diffmu}
  For any $n \geq 0$ and any $0 \leq s < t$ we have 
  \begin{equation}
    \|\mu_n(t, \cdot) - \mu_n(s, \cdot) - (t-s) (\cL^*
    \mu_n)(t,\cdot)\|_{\TV} = o(t-s) 
  \end{equation}
  where the norm is total variation and $(\cL^* \mu_n)(t,dq)$ is
  defined to be the signed kernel
  \begin{multline}
    \label{eqn:Lstarmu}
    \left[\frac{(\cL^0 \ell)(t, d\rho_0)}{\ell(t, d\rho_0)} +
      \sum_{i=1}^n \frac{(\cL^{\kappa} f)(t, \rho_{i-1},
        d\rho_i)}{f(t, \rho_{i-1}, d\rho_i)}\right] \\
    \times e^{-\lambda L} \bOne_{\Delta^L_n}(x_1,\ldots,x_n) \, dx_1
    \cdots dx_n \, \ell(t,d\rho_0) \prod_{j=1}^n f(t, \rho_{j-1},
    d\rho_j).
  \end{multline}
\end{lemma}

The expression for the measure above is to be understood formally; the
correct interpretation involves replacement, not division.  All of the
``divisors'' above are present as factors of $\ell(t, d\rho_0)
\prod_{j=1}^n f(t,\rho_{j-1}, d\rho_j)$, and the fractions indicate
that the appearance of the denominator in this portion is to be
replaced with the indicated numerator.

So that we know what to expect, before proceeding we note that when we
sum over $i$ in \eqref{eqn:Lstarmu}, some of the terms arising from
$\cL^0$ and $\cL^{\kappa}$ cancel.  Namely, the bracketed portion of
\eqref{eqn:Lstarmu} expands as
\begin{equation}
  \begin{aligned}
    &\frac{\int H[\rho_*,\rho_0] \ell(t, d\rho_*) f(t,\rho_*,
      d\rho_0)}{\ell(t, d\rho_0)} - \int H[\rho_0,\rho_*]
    f(t,\rho_0, d\rho_*) \\
    &+ \sum_{i=1}^n \left[\frac{\int (H[\rho_*,\rho_i] -
        H[\rho_{i-1},\rho_*]) f(t,\rho_{i-1}, d\rho_*)
        f(t,\rho_*, d\rho_i)}{f(t,\rho_{i-1}, d\rho_i)} \right. \\
    &\qquad \qquad \left. - \int (H[\rho_i,\rho_*] f(t,\rho_i,
      d\rho_*) - H[\rho_{i-1},\rho_*] f(t,\rho_{i-1}, d\rho_*))
    \right]
  \end{aligned}
\end{equation}
The gain terms associated with the kinetic equations we leave as they
are, but note that the ``loss'' terms telescope, and the above may be
shortened to
\begin{multline}
  \label{eqn:Lstarmu-short}
  \frac{\int H[\rho_*,\rho_0] \ell(t, d\rho_*) f(t,\rho_*,
    d\rho_0)}{\ell(t, d\rho_0)} - \int H[\rho_n,\rho_*]
  f(t,\rho_n,d\rho_*) \\
  + \sum_{i=1}^n \frac{\int (H[\rho_*,\rho_i] - H[\rho_{i-1},\rho_*])
    f(t,\rho_{i-1},d\rho_*)
    f(t,\rho_*,d\rho_i)}{f(t,\rho_{i-1},d\rho_i)}. 
\end{multline}
We are ready to prove our statistical characterization of the bounded
system.

\begin{proof}[Proof of Theorem~\ref{thm:bounded-sys}]
  For times $s$ and $t$ with $0 \leq s < t \leq T$, consider the
  difference $\int F(t,q) \, \mu(t,dq) - \int F(s,q) \, \mu(s,dq)$:
  \begin{multline}
    \label{eqn:differenceIandII}
     \int [F(t,q) - F(s,q)] \, \mu(t, dq) + \int F(t,q) [\mu(t, dq) -
     \mu(s, dq)] \\
    - \int [F(t,q) - F(s,q)] [\mu(t,dq) - \mu(s,dq)] = \mathrm{(I) +
      (II) + (III)}
  \end{multline}
  Since $F(t,q)$ is uniformly continuous in $t$ uniformly in $q$ by
  Lemma~\ref{lem:tcontinuity} and $\mu$ is a probability measure, (I)
  $\to 0$ as $t-s \to 0$.  Using Lemma~\ref{lem:diffmu} and the fact
  that $|F| \leq 1$, (II) $\to 0$ as $t-s \to 0$, and in fact 
  \begin{equation}
    \label{eqn:diffmuF}
    \int F(t,q) \frac{\mu(t,dq) - \mu(s,dq)}{t-s} \to \int F(t,q)(\cL^*
    \mu)(t,dq) 
  \end{equation}
  as $s \to t-$ with $t$ fixed.  Using both
  Lemmas~\ref{lem:tcontinuity} and \ref{lem:diffmu}, we see (III) is
  $o(t-s)$.  Thus $\int F(t,q) \, \mu(t,dq)$ is continuous in $t$;
  we will show additionally that it is differentiable from below in
  $t$ with one-sided derivative equal to $0$ for all $0 < t < T$.  In
  light of \eqref{eqn:diffmuF}, our task is to show that $-\int F(t,q)
  \, (\cL^* \mu)(t,dq)$ approximates (I) up to an $o(t-s)$ error.

  Using Lemma~\ref{lem:separate-random} we have the following
  approximation of the portion of (I) involving $\mu_n$, with error
  bounded by $C_{\lambda,H'(P)}[(t-s)^2 + (t-s) w(t-s)]$:
  \begin{equation}
    \begin{aligned}
      &\int \left[ F(t,q) (\bOne_B(q) + \bOne_V(q)) - F(t,\phi_s^t q)
        \left(\bOne_U(q) + \sum_{i=0}^{n-1} \bOne_{A_i}(q)\right)
      \right.\\
      &\qquad \qquad + (t-s) F(t, q) \int H[\rho_n,\rho_+]
      f(t,\rho_n,d\rho_+) \\
      &\qquad\qquad \qquad \left. - (t-s) \int \bbE F(t,
        \epsilon_{\rho_+} \phi_s^\tau q) H[\rho_n,\rho_+]
        f(t,\rho_n,d\rho_+) \vphantom{\sum_{i=0}^{n-1}} \right]
      \mu_n(t,dq).
    \end{aligned}
  \end{equation}
  We have $\int [F(t,q) \bOne_V(q) - F(t,\phi_s^t q) \bOne_U(q)] \,
  \mu_n(t,dq) = 0$ because the deterministic flow $\phi_s^t$ maps $U$
  to $V$ (modulo lower-dimensional sets) and preserve the Lebesgue
  measure in spatial coordinates.  Making replacements using
  Lemma~\ref{lem:boundary-approx} and reordering the terms, we find
  with an error bounded by $C_{L,n,\lambda,H'(P)}[(t-s) + w(t-s)]$
  that
  \begin{equation}
    \label{eqn:approx-diff-n}
    \begin{aligned}
      \int &\frac{F(t,q) - F(s,q)}{t-s} \mu_n(t,dq)  \\
      \approx&\int F(t,\epsilon_{\rho_+} q) H[\rho_{n-1},\rho_+]
      f(t,\rho_{n-1},d\rho_+) \mu_{n-1}(t,dq) \\
      & - \int \bbE F(t, \epsilon_{\rho_+} \phi_s^\tau q)
      H[\rho_n,\rho_+] f(t,\rho_n,d\rho_+) \mu_n(t,dq) \\
      &- \left[\int F(t,q) \left(\sum_{i=1}^{n-1}
          \frac{(H[\rho_*,\rho_i] - H[\rho_{i-1},\rho_*])
            f(t,\rho_{i-1},d\rho_*)
            f(t,\rho_*,d\rho_i)}{f(t,\rho_{i-1},d\rho_i)}\right.\right. \\
      &\left.\phantom{- \int F(t,q) \sum_{i=1}^{n-1}} +
        \frac{H[\rho_*,\rho_0] \ell(t,d\rho_*)
          f(t,\rho_*,d\rho_0)}{\ell(t,d\rho_0)} \right)
      \mu_{n-1}(t,dq) \\
      & \left. \qquad \qquad \qquad \phantom{\sum_{i=1}^{n-1}} - \int
        F(t,q) \int H[\rho_n,\rho_+] f(t,\rho_n,d\rho_+)
        \mu_n(t,dq)\right]
    \end{aligned}
  \end{equation}
  for positive integers $n$.  In the case $n = 0$, we have $\phi_s^t q
  = q$, and the approximation is
  \begin{equation}
    \label{eqn:approx-diff-0}
    \int \frac{F(t,q) - F(s,q)}{t-s} \mu_0(t,dq) \approx \int
    [F(t,\epsilon_{\rho_*} q) - F(t,q)] H[\rho_0,\rho_*]
    \mu_0(t,dq).
  \end{equation}
  We observe that the bracketed portion of \eqref{eqn:approx-diff-n}
  nearly matches \eqref{eqn:Lstarmu-short}, except that part involves
  $\mu_{n-1}$ and part involves $\mu_n$.  Furthermore, we have
  \begin{equation}
    \label{eqn:remove-tau}
    \begin{aligned}
      &\left|\int \bbE F(t, \epsilon_{\rho_+} \phi_s^\tau q) H[\rho_n,\rho_+]
        f(t,\rho_n,d\rho_+) \mu_n(t,dq) \right. \\
      &\qquad \left. - \int F(t, \epsilon_{\rho_+} q) H[\rho_n,\rho_+]
        f(t,\rho_n,d\rho_+) \mu_n(t,dq) \right| \leq C_{H'(P),n,L} (t-s).
    \end{aligned}
  \end{equation}
  This follows since $\tau$ is conditionally independent of $q$ given
  $\rho_n$, and $\tau \in (s,t)$ a.s.\ so that all but $O(t-s)$ volume
  of the $x$-simplex is simply translated by $\phi_s^\tau$ to a region
  of identical volume.  We make the replacement indicated by
  \eqref{eqn:remove-tau} in \eqref{eqn:approx-diff-n} without changing
  the form of the error.

  For any positive integer $N$ we define 
  \begin{equation}
    \Gamma_N(t) = \int F(t,q) \sum_{n=0}^N \mu_n(t,dq).
  \end{equation}
  Summing \eqref{eqn:approx-diff-n}, \eqref{eqn:approx-diff-0}, and
  our approximation for $\mu_n(t,\cdot) - \mu_n(s,\cdot)$ from
  \eqref{eqn:Lstarmu-short} gives
  \begin{equation}
    \begin{aligned}
      &\frac{\Gamma_N(t) - \Gamma_N(s)}{t-s} \\
      &\approx \int F(t,q) \left(\sum_{i=1}^N
        \frac{(H[\rho_*,\rho_i] - H[\rho_{i-1},\rho_*])
          f(t,\rho_{i-1},d\rho_*)
          f(t,\rho_*,d\rho_i)}{f(t,\rho_{i-1},d\rho_i)}\right. \\ 
      &\left.\phantom{- \int F(t,q) \sum_{i=1}^N} +
        \frac{H[\rho_*,\rho_0] \ell(t,d\rho_*)
          f(t,\rho_*,d\rho_0)}{\ell(t,d\rho_0)} \right)
      \mu_N(t,dq) \\
      &\quad - \int F(t,\epsilon_{\rho_+} q) H[\rho_N,\rho_+]
      f(t,\rho_N,d\rho_+) \mu_N(t,dq)
    \end{aligned}
  \end{equation}
  with an $o(t-s)$ error.  Call the right side $\gamma_N(t)$, so that
  the left derivative of $\Gamma_N$ is $\gamma_N$. Since $\Gamma_N$ is
  a Lipschitz function, we deduce
  \begin{equation}
    \Gamma_N(T) - \Gamma_N(0) = \int_0^T \gamma_N(t) \, dt.
  \end{equation}
  We have $\gamma_N(t)$ bounded uniformly in $t$ by
  \begin{equation}
    |\gamma_N(t)| \leq \frac{3H'(P)L^N \lambda^{N+1}}{(N-1)!},
  \end{equation}
  and thus $\gamma_N \to 0$ uniformly in $t$.  It follows that
  \begin{equation}
    \Gamma_N(T) - \Gamma_N(0) \to 0
  \end{equation}
  as $N \to \infty$.  Recognizing the limits of $\Gamma_N(T)$ and
  $\Gamma_N(0)$ as 
  \begin{equation}
    \int G(q) \mu(T,dq) \quad \text{and} \quad \int \bbE G(\Phi_0^T q)
    \mu(0,dq), 
  \end{equation}
  respectively, we have verified \eqref{eqn:unchanging} and completed
  the proof.
\end{proof}

Having shown the candidate measure $\mu(t,dq)$ constructed using the
solutions $\ell(t,d\rho_0)$ and $f(t,\rho_-,d\rho_+)$ given by Theorem
\ref{thm:kinetic}, our next task is to verify the latter, which we
undertake in the next section.

 

\section{The kinetic and marginal equation}
\label{sec:kinetic}

The primary goal of the present section is to prove
Theorem~\ref{thm:kinetic} concerning the kernel $f(t,\rho_-,d\rho_+)$
and marginal $\ell(t,d\rho_0)$, so that the candidate measure
$\mu(t,dq)$ in the previous section is well-defined and has the
properties required for the argument there.  

While we introduce our notation, we also discuss the intuitive meaning
of the terms of our kinetic equation, comparing with that of Menon and
Srinivasan \cite{Menon10}.  Let us write $\cL^{\kappa}$ for the
operator on $\cK$ given by the right-hand side of \eqref{eqn:kinetic},
\begin{equation}
  \label{eqn:kinetic-op}
  \begin{aligned}
    (\cL^{\kappa} k)(\rho_-,d\rho_+)
    &= \int (H[\rho_*,\rho_+] - H[\rho_-,\rho_*])
    k(\rho_-,d\rho_*) k(\rho_*,d\rho_+) \\
    &\quad - \left[\int H[\rho_+,\rho_*] k(\rho_+,d\rho_*) - \int
      H[\rho_-,\rho_*] k(\rho_-,d\rho_*)\right] k(\rho_-,d\rho_+),
  \end{aligned}
\end{equation}
so that the kinetic equation is $f_t = \cL^{\kappa} f$. 

The first term, which we call the \emph{gain} term, corresponds to the
production of a shock connecting states $\rho_-$ and $\rho_+$ by means
of collision of shocks connecting states $\rho_-,\rho_*$ and
$\rho_*,\rho_+$.  Such shocks have relative velocity given by 
\begin{equation}
  H[\rho_*,\rho_+] - H[\rho_-,\rho_*] = H[\rho_-,\rho_*,\rho_+]
  (\rho_+ - \rho_-).
\end{equation}
In the Burgers case, the second divided difference
$H[\cdot,\cdot,\cdot]$ is constant, and the above is proportional to
the sum of the increment $\rho_+ - \rho_-$, analogous to mass in the
Smoluchowski equation.

The second line of \eqref{eqn:kinetic-op} we call the ``loss''
term, though this need not be of definite sign.  To better understand
this, note that when we have proven Theorem~\ref{thm:kinetic}, we will
know that $f(t,\rho_-,d\rho_+)$ has total mass which is constant in
$(t,\rho_-)$.  In this case the loss term may be rewritten
equivalently as
\begin{multline}
  \label{eqn:loss-expanded}
  \left[ \int (H[\rho_+,\rho_*] - H[\rho_-,\rho_+])
    f(t,\rho_+,d\rho_*) \right. \\
  \left. - \int (H[\rho_-,\rho_*] - H[\rho_-,\rho_+])
    f(t,\rho_-,d\rho_*) \right] f(t,\rho_-,d\rho_+),
\end{multline}
which corresponds precisely with the kinetic equation of
\cite{Menon10}.  The meaning of the first line of
\eqref{eqn:loss-expanded} is clear: we lose a shock connecting states
$\rho_-,\rho_+$ when a shock connecting $\rho_+,\rho_*$ collides with
this, and the relative velocity is precisely $H[\rho_+,\rho_*] -
H[\rho_-,\rho_+]$.

The second line of \eqref{eqn:loss-expanded} is less easily
understood.  One would expect to find here a loss related to a shock
connecting $\rho_-,\rho_+$ colliding with $\rho_*,\rho_-$ for $\rho_*
< \rho_-$, particularly if we were viewing $f$ as a \emph{jump
  density} as \cite{Menon10} views its corresponding $n(t,y,dz)$.
Viewed as a \emph{rate kernel}, it is not clear that $f$ should suffer
such a loss: in some sense we must condition on being in state
$\rho_-$ for $f(t,\rho_-,\cdot)$ to be relevant at all.  At this point
we cannot offer an intuitive kinetic reason for the second line of
\eqref{eqn:loss-expanded}, but in light of the rigorous results of
this article we can be assured that this is correct for our model.

\begin{remark}
  The form \eqref{eqn:loss-expanded} seems preferable in the more
  generic setting, since---as pointed out by the referee---the
  resulting equation $f_t = \cL^\kappa f$ has the property that 
  \begin{equation}
    \lambda(t,\rho_-) = \int f(t,\rho_-,d\rho_+)
  \end{equation}
  is (formally) conserved in time for each $\rho_-$, \emph{without
    assuming that $\lambda(0,\rho_-)$ is constant.}  This observation
  will be important in attempts to generalize the results of the
  present paper.
\end{remark}

We return to the task at hand, showing existence and uniqueness of
$f(t,\rho_-,d\rho_+)$.  The argument proceeds through an approximation
scheme for $\exp(\lambda H'(P) t) f(t,\cdot,\cdot)$, where we can
easily maintain positivity.  To avoid endowing $\cK$ with a weak
topology, we show directly the approximations are Cauchy, rather than
appealing to Arzela-Ascoli.

\begin{proof}[Proof of Theorem~\ref{thm:kinetic}, part I]
  Write $c = \lambda H'(P)$ and consider for each positive integer $n$
  the continuous paths $h^n : [0,\infty) \to \cK$ defined by $h^n(0) =
  g$ for all $n$ and 
  \begin{equation}
    \dot{h}^n(t) = e^{-cj/n} (\cL^{\kappa} h^n)(j/n) + c
    h^n(j/n), \qquad t \in (j/n, (j+1)/n).
  \end{equation}
  We claim the following properties of $h^n$: 
  \begin{itemize}
  \item[(a)] $h^n(t) \geq 0$ for all $t \geq 0$,
  \item[(b)] for each $t \geq 0$ the total integral $\int
    h^n(t,\rho_-,d\rho_+)$ is constant in $\rho_- \in [0,P)$,
  \item[(c)] for all $t \geq 0$ 
    \begin{equation}
      \label{eqn:hn-bound}
      \|h^n(t)\| \leq \lambda (1+c/n)^{\lceil nt \rceil} \leq
      \lambda e^{c(t+1)} =: M(t),
    \end{equation}
    and
  \item[(d)] $h^n(t,\rho_-,[0,\rho_-] \cup \{P\}) = 0$ for all $t$
    and all $\rho_- \in [0,P)$.
  \end{itemize}
  Since $h^n$ is piecewise-linear in $t$ and the properties above are
  preserved by convex combinations, it suffices to verify this at $t =
  j/n$ for all integers $j \geq 0$.  

  We proceed by induction.  Abbreviate $h^n_j = h^n( j/n)$.  The
  $j = 0$ case holds by the hypotheses on $g$.  Assume now that the
  claim holds for $j$, and consider $j+1$.  We have
  \begin{equation}
    \label{eqn:iter-change}
    \begin{aligned}
      n (h^n_{j+1} - h^n_j) 
      &= e^{-cj/n} \left[\int (H[\rho_*,\rho_+] - H[\rho_-,\rho_*])
        h^n_j(\rho_-,d\rho_*) h^n_j(\rho_*,d\rho_+) \right. \\
      &\qquad \qquad \qquad + \left(\int H[\rho_-,\rho_*]
        h^n_j(\rho_-,d\rho_*)\right) h^n_j(\rho_-,d\rho_+) \\
      &\qquad \qquad \qquad \left. - \left(\int H[\rho_+,\rho_*]
        h^n_j(\rho_+,d\rho_*)\right) h^n_j(\rho_-,d\rho_+)\right] \\
      &\qquad + c h^n_j(\rho_-,d\rho_+).
    \end{aligned}
  \end{equation}
  In particular, for any $\rho_+$ we have
  \begin{equation}
    \begin{aligned}
      c - e^{-cj/n} \int H[\rho_+,\rho_*] h^n_j(\rho_+,d\rho_*)
      &\geq c - e^{-cj/n} H'(P) \|h^n_j\| \\
      &\geq c - e^{-cj/n} H'(P) \lambda (1+c/n)^j \\
      &= c\left[1 - e^{-cj/n}(1+c/n)^j\right] \geq 0,
    \end{aligned}
  \end{equation}
  using property (c) for case $j$; thus $n(h^n_{j+1} - h^n_j) \geq 0$.
  This and (a) $h^n_j \geq 0$ give $h^n_{j+1} \geq 0$, property (a)
  for case $j+1$.  Furthermore, the total integral of the right-hand
  side of \eqref{eqn:iter-change} is exactly
  \begin{equation}
    c \int h^n_j(\rho_-,d\rho_+) = c \|h^n_j\|,
  \end{equation}
  since (b) implies the bracketed portion integrates to $0$.  Also
  using (b) for case $j$, the change in the total integral from
  $h^n_j$ to $h^n_{j+1}$ is independent of $\rho_-$, verifying (b) for
  case $j+1$.  Using (c) for case $j$ we find 
  \begin{equation}
    \|h^n_{j+1}\| \leq \|h^n_j\| +  c n^{-1}\|h^n_j\| = (1 + c/n)
    \|h^n_j\| \leq \lambda (1 +  c/n)^j,
  \end{equation}
  so that (c) holds for case $j+1$.  Finally, we observe that property
  (d) for $h^n_j$ implies 
  \begin{equation}
    \cL^\kappa h^n_j(\rho_-,[0,\rho_-] \cup \{P\}) = 0,
  \end{equation}
  and thus (d) holds for case $j+1$.

  We now consider the matter of convergence.  Pairing $\cL^{\kappa}$
  with any $k_1,k_2 \in \cK$ and measurable $J(\rho_+)$ with $|J| \leq
  1$, we find that
  \begin{equation}
    \label{eqn:kinetic-diff}
    \|\cL^{\kappa} k_1 - \cL^{\kappa} k_2\| \leq 3 H'(P) (\|k_1\| +
    \|k_2\|) \|k_1 - k_2\|. 
  \end{equation}
  In particular $\cL^\kappa$ is Lipschitz on bounded sets in $\cK$.
  For brevity write 
  \begin{equation}
    \cL^h k(s,\rho_-,d\rho_+) = e^{-cs} \cL^\kappa k(\rho_-,d\rho_+) +
    c k(\rho_-,d\rho_+).
  \end{equation}
  We compute for any $m<n$ and $t \geq 0$
  \begin{multline}
    \label{eqn:show-hn-cauchy}
    \|h^{m}(t) - h^n(t)\|
    \leq \int_0^t \|\dot{h}^{m}(s) - \cL^h h^{m}(s)\| + \|\cL^h
    h^{m}(s) - \cL^h h^n(s)\| \\
    + \|\cL^h h^n(s) - \dot{h}^n(s)\| \, ds.
  \end{multline}
  Observe that 
  \begin{equation}
    \begin{aligned}
      \|\dot{h}^n(s) - \cL^h h^n(s)\|
      &= \|\cL^h h^n(n^{-1} \lfloor n s \rfloor) - \cL^h h^n(s)\| \\
      &\leq \|e^{-cn^{-1} \lfloor n s \rfloor} \cL^\kappa h^n(n^{-1}
      \lfloor n s \rfloor) - e^{-cs} \cL^\kappa h^n(s)\| \\
      &\qquad + c\|h^n(n^{-1} \lfloor n s \rfloor) - h^n(s)\|
    \end{aligned}
  \end{equation}
  where 
  \begin{equation}
    \begin{aligned}
      \|e^{-cn^{-1} \lfloor n s \rfloor} & \cL^\kappa h^n(n^{-1}
      \lfloor n s \rfloor) - e^{-cs} \cL^\kappa h^n(s)\| \\
      &\leq e^{-cs} \|\cL^\kappa h^n(n^{-1} \lfloor n s \rfloor) -
      \cL^\kappa h^n(s)\| \\
      &\qquad + (e^{-cn^{-1} \lfloor n s \rfloor} - e^{-cs})
      \|\cL^\kappa h^n(n^{-1} \lfloor n s \rfloor)\| \\
      &\leq e^{-cs} 6 H'(P) M(s) \|h^n(n^{-1} \lfloor n s \rfloor) -
      h^n(s)\| \\
      &\qquad + e^{-cs} e^{c n^{-1}} n^{-1} 3 H'(P) M(s)^2 \\
      &\leq 6 c e^c \|h^n(n^{-1} \lfloor n s \rfloor) -  h^n(s)\| +
      3 c \lambda e^{cs + 3c}n^{-1},
    \end{aligned}
  \end{equation}
  We have also 
  \begin{equation}
    \begin{aligned}
      \|h^n(n^{-1} \lfloor n s \rfloor) - h^n(s)\|
      &\leq n^{-1} \|\cL^h h^n(n^{-1} \lfloor n s \rfloor)\| \\
      &\leq n^{-1} (e^{-cs} 3 H'(P) M(s)^2 + M(s)) \\
      &= n^{-1} (3 \lambda c e^{cs + 2c} + \lambda e^{c(s+1)}).
    \end{aligned}
  \end{equation}
  Combining these, there is a constant $C_1$ so that
  \begin{equation}
    \|\dot{h}^n(s) - \cL^h h^n(s)\| \leq C_1 n^{-1} e^{cs}
  \end{equation}
  for all positive integers $n$ and times $s \geq 0$.

  The remaining term in the integrand of \eqref{eqn:show-hn-cauchy} is
  \begin{equation}
    \begin{aligned}
      \|\cL^h h^{m}(s) - \cL^h h^n(s)\|
      &\leq e^{-cs} \|\cL^\kappa h^{m}(s) - \cL^\kappa h^n(s)\| + c
      \|h^{m}(s) - h^n(s)\| \\
      &\leq (e^{-cs} 6H'(P) M(s) + c) \|h^{m}(s) - h^n(s)\| \\
      &= (6 e^c + c) \|h^{m}(s) - h^n(s)\| = C_2 \|h^{m}(s) -
      h^n(s)\|.
    \end{aligned}
  \end{equation}
  All together,
  \begin{equation}
    \|h^{m}(t) - h^n(t)\| \leq \int_0^t C_2 \|h^m(s) - h^n(s)\|
    + C_1 (n^{-1} + m^{-1}) e^{cs} \, ds,
  \end{equation}
  and by Gronwall 
  \begin{equation}
    \|h^{m}(t) - h^n(t)\| \leq C_1 (n^{-1} + m^{-1}) e^{C_2 t} \frac{e^{ct}
      - 1}{c}.
  \end{equation}
  From this we see that $h^n$ is Cauchy,
  hence convergent to a continuous $h : [0,\infty) \to \cK$ satisfying
  \begin{equation}
    \label{eqn:h-solves}
    h(t) = g + \int_0^t [e^{-cs} (\cL^\kappa h)(s) + ch(s)] \, ds
  \end{equation}
  for all $t \geq 0$, having properties (a,b,d) above and
  $\|h(t)\| \leq M(t)$.

  We define $f(t) = e^{-ct} h(t)$, finding that 
  \begin{itemize}
  \item $t \mapsto f(t) \in \cK_+$ is continuous;
  \item $t \mapsto \cL^\kappa f(t) \in \cK$ is continuous (using the
    local Lipschitz property of $\cL^\kappa$), and hence
    Bochner-integrable; and
  \item $\dot{f} = \cL^\kappa f$ with $f(0) = g$ by Leibniz.
  \end{itemize}
  Using the local Lipschitz property of $\cL^\kappa$, the solution is
  unique.  Properties (a,b,d) above hold for $f$, and (b) in
  particular gives for each fixed $\rho_-$
  \begin{equation}
    \frac{d}{dt} \int f(t,\rho_-,d\rho_+) = \int (\cL^\kappa
    f)(t,\rho_-,d\rho_+) = 0.
  \end{equation}
  Thus $\int f(t,\rho_-,d\rho_+) = \lambda$ for all $t \geq 0$, $0
  \leq \rho_- < P$.
\end{proof}

We now turn to $\ell(t,d\rho_0)$, which is intended to serve as the
marginal at $x = 0$ for our solution $\rho(x,t)$ for fixed $t > 0$.
We define a time-dependent family of operators $\cL^0$ acting on
measures $\nu(d\rho_0)$ in $\cM$ by
\begin{multline}
  \label{eqn:marginal-forward}
  (\cL^0 \nu)(t, d\rho_0) = \int H[\rho_*,\rho_0] \nu(d\rho_*)
  f(t,\rho_*,d\rho_0) \\
  - \left[\int H[\rho_0,\rho_*] f(t,\rho_0,d\rho_*)\right]
  \nu(d\rho_0). 
\end{multline}
Again the integration is over $\rho_*$ only, and for each $t$ we have
$(\cL^0 \nu)(t,\cdot) \in \cM$.  Our evolution equation for
$\ell(t,d\rho_0)$ is the linear, time-inhomogeneous $\ell_t = \cL^0
\ell$.

\begin{proof}[Proof of Theorem~\ref{thm:kinetic}, part II]
  Note that \eqref{eqn:marginal-forward} is exactly the forward
  equation for a pure-jump Markov process evolving according to a
  time-varying rate kernel 
  \begin{equation}
    H[\rho_0,\rho_+] f(t,\rho_0,d\rho_+),
  \end{equation}
  so we could obtain existence and uniqueness of solutions along these
  lines.  We outline an argument similar to that employed for $f$, for
  the sake of completeness.

  Define for each $n$ the continuous path $h^n : [0,\infty) \to \cM$
  with $h^n(0) = \delta_0$ and
  \begin{equation}
    \dot{h}^n(t) = e^{-cj/n} \cL^0 h^n(j/n) + c h^n(j/n), \qquad t \in
    (j/n,(j+1)/n).
  \end{equation}
  Proceeding as in the proof for $f$ we verify that for each $j$ the
  $h^n_j = h^n(j/n)$ are nonnegative and, since $\cL^0 h^n(j/n)$ has
  total integral zero, $\|h^n_j\| \leq (1+c/n)^{\lceil n t \rceil}$.
  Observing that $\cL^0$ is linear and bounded uniformly over $t$,
  \begin{equation}
    \label{eqn:Lzero-bounded}
    \sup_t \|(\cL^0 \nu)(t,\cdot)\|_{\TV} \leq 2 \lambda H'(P)
    \|\nu\|_{\TV}, 
  \end{equation}
  we easily show $h^n$ is Cauchy on bounded time intervals, with limit
  $h : [0,\infty) \to \cM_+$ satisfying
  \begin{equation}
    h(t) = \delta_0 + \int_0^t [e^{-cs} (\cL^0 h)(s) + ch(s)] \, ds.
  \end{equation}
  Take $\ell(t) = e^{-ct} h(t)$ to find that $\ell(t) \geq 0 $ solves
  \begin{equation}
    \ell(t) = \delta_0 + \int_0^t (\cL^0 \ell)(s) \, ds.
  \end{equation}
  Uniqueness follows using \eqref{eqn:Lzero-bounded}.  Finally, since
  $\cL^0 \nu$ as zero total integral for any $\nu$, we find that $\int
  \ell(t,d\rho_0) = 1$ for all $t$.
\end{proof}

We close this section with a remark concerning the kinetic equation
$f_t = \cL^\kappa f$ without the assumption that the initial kernel
$g$ has bounded support $[0,P]$.  When the initial condition $\xi(x)$
is unbounded, growing for example linearly in the case of quadratic
$H$, the solution to the scalar conservation law $\rho_t = H(\rho_x)$
on the semi-infinite domain should blow up in finite time.  We
likewise expect that the solution $f$ to the kinetic equation will
blow up, but control of certain moments prior to this blow up may
allow us to run the argument of Theorem~\ref{thm:bounded-sys} with
only superficial changes.  At the moment we lack the sort of estimates
one obtains in the Smoluchowski case (e.g.\ coming from a closed
equation for the second moment), but we hope to revisit this in future
work.



\section{Conclusion}
\label{sec:conclusion}

We review what has been accomplished: using an exact propagation of
chaos calculation for a bounded system with a suitably selected random
boundary condition, we have derived a complete description of the law
of the solution $\rho(x,t)$ to the scalar conservation law $\rho_t =
H(\rho)_x$ for $x,t > 0$ when $\rho(x,0) = \xi(x)$ is a bounded,
monotone, pure-jump Markov process with constant jump rate.  Notably,
we have recovered the Markov property of the solution and a
statistical description of the shocks \emph{simultaneously}.  This may
be regarded as both a strength of our present analysis and a
shortcoming of our present understanding (a soft argument for the
preservation of the Markov property in the particle system would be
illuminating).

We emphasize how our approach using random dynamics on a bounded
interval has made things considerably easier: by constructing a
bounded system which has \emph{exactly} the right law, we are relieved
of the burden of determining precisely how wrong the law would be with
deterministic boundary.  

Our result lends additional support to the conjecture of Menon and
Srinivasan \cite{Menon10} and we hope that the sticky particle methods
described herein will provide another approach, quite different from
that of Bertoin \cite{Bertoin98}, which might be adapted to resolve
the full conjecture.  One of the authors is attempting a similar
particle approach in the non-monotone setting, and hopes to report on
this in future work.

\subsection*{Acknowledgments}

The authors thank the anonymous referee for suggestions improving our
presentation and for pointing out a recent reference.



\appendix

\section{Proofs of lemmas}
\label{app:proofs}

\begin{proof}[Proof of Lemma~\ref{lem:tcontinuity}]
  Boundedness of $G$ is obvious from its definition as a Laplace
  functional.  For continuity, we begin by considering $G(\phi_0^t q)$
  with $q$ fixed.  As $t$ varies, $\phi_0^t q$ varies continuously,
  with $\rho$-values unchanging and $x$-values changing with rate
  bounded by $H'(P)$, \emph{except} for finitely many times when
  collisions occur.  Note $G(\phi_0^t q)$ depends on $\phi_0^t q$ via
  $\pi(\phi_0^t q, \cdot)$, and the former varies continuously across
  these collision times.  More concretely, because $J \in C([0,L])$,
  $J$ is uniformly continuous, and so
  \begin{equation}
    w_J(\delta) = \sup \{|J(x) - J(y)| : x,y \in [0,L], |x - y| \leq
    \delta\}
  \end{equation}
  has the property that $w_J(\delta) \to 0$ as $\delta \to 0+$.
  Using the above,
  \begin{equation}
    |\log G(\phi_0^t q) - \log G(q)| \leq \sum_{i=1}^n (\rho_i -
    \rho_{i-1}) w_J(t H'(P)) \leq P w(t H'(P)).
  \end{equation}
  Since the exponential is uniformly continuous for nonpositive
  arguments, we find that $G(\phi_0^t q)$ is continuous at $t = 0$
  uniformly in $q$.  Write 
  \begin{equation}
    w_G(\delta) = \sup \{|G(\phi_0^t q) - G(q)| : q \in Q, 0 \leq t \leq
    \delta\},
  \end{equation}
  and observe that $w_G(\delta) \to 0$ and $\delta \to 0+$.

  Suppose now that $0 \leq s < t \leq T$; we compare $\bbE G(\Phi_s^T
  q)$ and $\bbE G(\phi_t^T q)$.  Write $\theta = t - s$.  We have
  $G(\Phi_s^T q) = G(\Phi_{T-\theta}^T \Phi_s^{T-\theta} q)$ a.s., coupling
  the random entry processes on matching intervals.  Then 
  \begin{equation}
    \bbE |G(\Phi_{T-\theta}^T \Phi_s^{T-\theta} q) -
    G(\phi_{T-\theta}^T \Phi_s^{T-\theta} q)| \leq C_{\lambda,H'(P)} \theta,
  \end{equation}
  because the random entries occur at a rate bounded by $\lambda H'(P)$
  and $|G| \leq 1$.  Using the time homogeneity and continuity for the
  deterministic flow,
  \begin{equation}
    \bbE |G(\phi_{T-\theta}^T \Phi_s^{T-\theta} q) -
    G(\Phi_s^{T-\theta} q)| \leq w_G(\theta).
  \end{equation}
  So, at the cost of an error bounded by $C_{\lambda,H'(P)} \theta +
  w_G(\theta)$, we compare $\bbE G(\Phi_t^T q)$ with $\bbE
  G(\Phi_s^{T-\theta} q)$ instead.  Now the configuration $q$ is
  flowed on time intervals of equal length, but with random entry
  rates that have been shifted in time.

  The random entry rate is given by $H[\rho_n,\rho_+] f(r, \rho_n,
  d\rho_+)$ and $f$ is TV-continuous in $r$; we have
  \begin{equation}
    \begin{aligned}
      \|f(r, \rho_n, d\rho_+) - f(r-\theta, \rho_n, d\rho_+)\| &=
      \left\|\int_{r-\theta}^r (\cL^{\kappa} f)(\tau, \rho_n, d\rho_+)
        \, d\tau \right\| \\
      &\leq C_{\lambda, H'(P)} \theta.
    \end{aligned}
  \end{equation}
  Let us define three kernels for $r \in [0,T-t]$ according to
  \begin{equation}
    \begin{aligned}
      \hat{f}(r, \rho_-, d\rho_+) &= f(s + r, \rho_-, d\rho_+) \wedge
      f(t + r, \rho_-, d\rho_+) \\
      f^s(r, \rho_-, d\rho_+) &= f(s + r, \rho_-, d\rho_+) -
      \hat{f}(r, \rho_-, d\rho_+) \\
      f^t(r, \rho_-, d\rho_+) &= f(t + r, \rho_-, d\rho_+) -
      \hat{f}(r, \rho_-, d\rho_+),
    \end{aligned}
  \end{equation}
  where the minimum ($\wedge$) of two measures is defined as usual by
  choosing a third measure with which the former two are absolutely
  continuous, and taking the pointwise minimum of their Radon-Nikodym
  derivatives.  That this extends measurably to the parametric case
  (i.e.~involving kernels) is immediate from a parametric version of
  the Radon-Nikodym theorem \cite[Theorem 2.3]{Novikov05}.

  Using the above we construct a coupled random entry process for
  times $r \in [0,T-t]$, namely let $(\zeta^s,\zeta^t)(r)$ be the pure-jump
  Markov process started at $(\rho_n,\rho_n)$ from $q$ and evolving
  according to the generator $\cL^{z}$ acting on bounded measurable
  functions $J(y,z)$ defined on $[0,P]^2$ according to
  \begin{equation}
    \label{eqn:coupledentry}
    \begin{aligned}
      (\cL^{z} J)(r,y,z) = \quad &\bOne(y = z) \int
      \{J(\rho_+,\rho_+) - J(y,z)\} H[y,\rho_+] \hat{f}(r,y,d\rho_+) \\
      \quad + &\bOne(y = z) \int 
      \{J(\rho_+,z) - J(y,z)\} H[y,\rho_+] f^s(r,y,d\rho_+)\\
      \quad + &\bOne(y = z) \int 
      \{J(y,\rho_+) - J(y,z)\} H[z,\rho_+] f^t(r,z,d\rho_+)\\
      \quad + &\bOne(y \neq z) \int 
       \{J(\rho_+,z) - J(y,z)\} H[y,\rho_+] (\hat{f} + f^s)(r,y,d\rho_+)\\
      \quad + &\bOne(y \neq z) \int
       \{J(y,\rho_+) - J(y,z)\} H[z,\rho_+] (\hat{f} + f^t)(r,z,d\rho_+).
    \end{aligned}
  \end{equation}
  Taking $J(y,z)$ which does not depend on $z$ we find that $\zeta^s(r)$
  for $r \in [0,T-t]$ has the same law as the random boundary for
  $\Phi_s^{T-\theta} q$, and likewise taking $J$ which does not depend
  on $y$ we find the $\zeta^t(r)$ has the same law as the random boundary
  for $\Phi_t^T q$, verifying the coupling.  

  On the diagonal $y = z$, the rate at which $\cL^{z}$ causes jumps
  that leave the diagonal (the second and third lines of
  \eqref{eqn:coupledentry}) is bounded by
  \begin{equation}
    C_{\lambda,H'(P)} \theta,
  \end{equation}
  and the probability that such a transition never occurs in a time
  interval of length $T - t$ is bounded below by
  \begin{equation}
    \exp[-C_{\lambda,H'(P)} (T-t) \theta] \geq
    \exp[-C_{\lambda,H'(P)} T \theta]. 
  \end{equation}
  So, coupling the random entry dynamics, we find that
  $\Phi_s^{T-\theta} q = \Phi_t^T q$ with probability at least
  $\exp[-C_{\lambda,H'(P)} T \theta]$.  Putting the above pieces
  together, we find
  \begin{multline}
    |\bbE G(\Phi_s^T q) - \bbE G(\Phi_t^T q)| \\
    \leq C_{\lambda,H'(P)}(t-s) + w_G(t-s) + (1 -
    \exp(C_{\lambda,H'(P)} T (t-s))),
  \end{multline}
  and the proof is complete.
\end{proof}

\begin{proof}[Proof of Lemma~\ref{lem:separate-random}]
  Using the Markov property of the random flow $\Phi$, we have a
  functional identity:
  \begin{equation}
    F(s,q) = \bbE G(\Phi_t^T \Phi_s^t q) = \bbE F(t, \Phi_s^t q).
  \end{equation}
  Let $q$ be fixed, and consider the following events, whose union is
  of full measure for computing the final expectation above:
  \begin{equation}
    \begin{aligned}
      \cE_0 &= \{\text{no entry at $x = L$ in $(s,t)$}\} \\
      \cE_1 &= \{\text{at least one entry at $x = L$ in $(s,t)$}\}
    \end{aligned}
  \end{equation}
  Observe that on $\cE_0$ we see only the deterministic flow $\phi$ over
  the time interval $(s,t)$:
  \begin{equation}
    F(t,\Phi_s^t q) \bOne_{\cE_0} = F(t,\phi_s^t q) \bOne_{\cE_0},
  \end{equation}
  and this occurs with probability
  \begin{equation}
    \bbP(\cE_0) = \exp\left(-\int_s^t \int H[\rho_n, \rho_+] f(r,
      \rho_n; d\rho_+) \, dr \right).
  \end{equation}
  We prefer an expression evaluating $f$ at only a single time, and
  Taylor expand the exponential around zero.
  \begin{multline}
    \label{eqn:homogeneousrate}
    \left|\int_s^t \int H[\rho_n,\rho_+] f(r, \rho_n, d\rho_+) \, dr
      - (t-s) \int  H[\rho_n,\rho_+] f(t, \rho_n, d\rho_+)\right| \\
    = \left|\int_s^t \int_t^r \int H[\rho_n,\rho_+] (\cL^{\kappa}
      f)(r', \rho_n, d\rho_+) \, dr' \, dr \right| \leq C_{\lambda,
      H'(P)} (t-s)^2,
  \end{multline}
  so that 
  \begin{equation}
    \bbP(\cE_0) = 1 - (t-s) \int H[\rho_n,\rho_+] f(t, \rho_n, d\rho_+) 
    + O((t-s)^2)
  \end{equation}
  and by exhaustion
  \begin{equation}
    \bbP(\cE_1) = (t-s) \int H[\rho_n,\rho_+] f(t, \rho_n, d\rho_+)  +
    O((t-s)^2),
  \end{equation}
  with both errors bounded uniformly over $Q$.  On $\cE_1$ write $\tau$
  for the first time a random entry occurs for $\Phi_s^t$, and
  $\rho_+$ for the new boundary value, noting that the distribution of
  $\tau$ depends only on $q$ through $\rho_n$ and that the law of
  $\rho_+$ is determined by $f(\tau,\rho_n, \cdot)$.  We have $\tau
  \in (s,t)$ so that
  \begin{equation}
    F(t,\Phi_s^t q) \bOne_{\cE_1} = F(t,\Phi_\tau^t \epsilon_{\rho_+}
    \phi_s^\tau q) \bOne_{\cE_1}.
  \end{equation}
  Using the strong Markov property for the
  random boundary at the stopping time $\tau$,
  \begin{equation}
    \bbE F(t,\Phi_\tau^t \epsilon_{\rho_+} \phi_s^\tau q) \bOne_{\cE_1}
    = \bbE F(\tau, \epsilon_{\rho_+} \phi_s^\tau q) \bOne_{\cE_1}.
  \end{equation}
  Since $\bbP(\cE_1) = O(t-s)$, we can afford to make $o(1)$
  modifications to this. Write $w(\delta)$ for the modulus of
  continuity of $F(t,q)$ in time, according to
  Lemma~\ref{lem:tcontinuity}.  Now $\tau \in (s,t)$ a.s.\ on $\cE_1$,
  so
  \begin{equation}
    |\bbE F(\tau, \epsilon_{\rho_+} \phi_s^\tau q) \bOne_{\cE_1} -
    \bbE F(t, \epsilon_{\rho_+} \phi_s^\tau q) \bOne_{\cE_1}| \leq
    \bbP(\cE_1) w(t-s). 
  \end{equation}

  Next we modify the distribution from which $\rho_+$ is selected; at
  present, $\rho_+$ is selected according to the random measure
  \begin{equation}
    \label{eqn:law-rhoplus}
    \frac{ H[\rho_n,\rho_+] f(\tau,\rho_n, d\rho_+) }{\int
      H[\rho_n,\rho_+] f(\tau,\rho_n, d\rho_+) }.
  \end{equation}
  Since $\tau \in (s,t)$ a.s.\ on $\cE_1$, the total variation difference 
  \begin{equation}
    \|f(\tau,\rho_n, d\rho_+) - f(t,\rho_n, d\rho_+)\| \leq
    C_{\lambda,H'(P)} (t-s) \quad \text{a.s.}
  \end{equation}
  Let us write $\hat{\rho}_+$ for an independent random variable
  distributed as 
  \begin{equation}
  H[\rho_n,\rho_+] f(t,\rho_n, d\rho_+),
  \end{equation}
  (note $t$  instead of $\tau$), normalized to have unit mass.  Then
  \begin{equation}
    \left|\bbE F(t,\epsilon_{\rho_+} \phi_s^\tau q) \bOne_{\cE_1} - \bbE
      F(t, \epsilon_{\hat{\rho}_+} \phi_s^\tau q) \bOne_{\cE_1} \right| 
    \leq C_{\lambda,H'(P)} (t-s) \bbP(\cE_1).
  \end{equation}
  Note that $\bbP(\cE_1)$ cancels with the normalization in
  \eqref{eqn:law-rhoplus}, up to an $O(t-s)$ error, and that
  $\hat{\rho}_+$ is independent of $\cE_1$.  With error bounded
  uniformly over $q \in Q$ by $C_{\lambda,H'(P)}[(t-s)^2 + (t-s)
  w(t-s)]$, we have
  \begin{multline}
    \label{eqn:approxexpF}
    \bbE F(t,\Phi_s^t q) \approx F(t,\phi_s^t q) \left(1 - (t-s) \int
      H[\rho_n,\rho_+] f(t,\rho_n, d\rho_+) \right) \\
    + (t-s) \int \bbE[F(t, \epsilon_{\rho_+} \phi_s^\tau q) \mid \cE_1]
    H[\rho_n,\rho_+] f(t,\rho_n,d\rho_+),
  \end{multline}
  where the remaining expectation is now taken only over $\tau$
  conditioned on $\cE_1$, which is therefore distributed over $[s,t]$.
\end{proof}

\begin{proof}[Proof of Lemma~\ref{lem:boundary-approx}]
  On any of $A_i$, $i = 0,\ldots,n-1$, we have $\sigma = \sigma(q)
  \leq t$.  For such $q$ we have
  \begin{multline}
    \label{eqn:replace-boundary}
    |F(t,\phi_s^t q) - F(t,\phi_s^\sigma q)| \\
    \leq |F(t,\phi_\sigma^t \phi_s^\sigma q) -
    F(\sigma,\phi_s^\sigma q)| + |F(\sigma, \phi_s^\sigma q) -
    F(t,\phi_s^\sigma q)|.
  \end{multline}
  The first absolute difference is the error induced by conditioning
  on zero particle entries on the interval $[s,\sigma]$, which costs
  less than $\lambda H'(P) (t-s)$.  Lemma~\ref{lem:tcontinuity} bounds
  the second of these by $w(t-s)$.  Using \eqref{eqn:replace-boundary}
  we can replace the integrand $F(t,\phi_s^t q)$ over $A_i$ with
  $F(t,\phi_s^\sigma q)$, i.e.\ the projection of $q$ onto
  $(\partial_i \Delta^L_n) \times \overline{\Delta^P_{n+1}}$ in the
  direction of the velocity.
  
  Write $v_i = -H[\rho_{i-1},\rho_i]$ for the velocities.  We can
  integrate over $A_0$ using as a parametrization for the
  $x$-variables
  \begin{equation}
    (\theta,x_2,\ldots,x_n) \mapsto (-\theta v_1, x_2 - \theta v_2,
    \ldots, x_n - \theta v_n),
  \end{equation}
  which has absolute Jacobian determinant $-v_1 = H[\rho_0,\rho_1]$.
  Namely, 
  \begin{equation}
    \begin{aligned}
      &\int_{A_0} F(t,\phi_s^\sigma q) \mu_n(t,dq) \\
      &\quad =  \int_{\Delta^L_{n-1} \times \overline{\Delta^P_{n+1}}}
      \int_0^{(t-s) \wedge |(L-x_n)/v_n|}
      F(t,(0,x_2,\ldots,x_n;\rho_0,\ldots,\rho_n)) \\
      &\hspace{13em} H[\rho_0,\rho_1]\ d\theta \, dx_2 \cdots dx_n \,
      \ell(t,d\rho_0) \prod_{i=1}^n f(t,\rho_{i-1},d\rho_i).
    \end{aligned}
  \end{equation}
  On all but a portion of $A_0$ with $x$-volume bounded by
  $C_{n,L,H'(P)} (t-s)^2$ we have $t - s < |(L - x_n)/v_n|$, and we
  can extend the upper limit of integration to $\theta = t-s$ with
  and error bounded by this.  Noting that the configurations 
  \begin{equation}
    (0,x_2,\ldots,x_n; \rho_0,\ldots,\rho_n) \text{ and }
    (x_2,\ldots,x_n; \rho_1,\ldots,\rho_n)
  \end{equation}
  are equivalent under our sticky particle dynamics and relabeling, we
  obtain our approximation for the integral over $A_0$.

  The argument for $A_i$, $i = 1,\ldots,n-1$, is similar.  We use the
  mapping 
  \begin{equation}
    (x_1,\ldots,x_i,x_{i+2},\ldots,x_n,\theta) \mapsto
    (x_1 - \theta v_1,\ldots, x_i - \theta v_i, x_i - \theta v_{i+1},
    \ldots, x_n - \theta v_n),
  \end{equation}
  with absolute Jacobian determinant $v_i - v_{i+1} =
  H[\rho_i,\rho_{i+1}] - H[\rho_{i-1},\rho_i]$, and replace the upper
  limit of integration for $\theta$ with $t-s$ at the cost of an
  $O((t-s)^2)$ error.  Relabeling gives the indicated approximation.

  The range of 
  \begin{equation}
    \phi_s^{s+\theta}(x_1,\ldots,x_{n-1},L; \rho_0,\ldots,\rho_n)
  \end{equation}
  where $(x_1,\ldots,x_{n-1}) \in \Delta^L_{n-1}$,
  $(\rho_0,\ldots,\rho_n) \in \overline{\Delta^P_{n+1}}$, and 
  \begin{equation}
    0 \leq \theta \leq (t-s) \wedge \frac{x_1}{|v_1|} \wedge \frac{x_2
      - x_1}{|v_2 - v_1|} \wedge \cdots \wedge \frac{x_n -
      x_{n-1}}{|v_n - v_{n-1}|}
  \end{equation}
  is $B$, modulo sets of lower $x$-dimension.  Using
  Lemma~\ref{lem:tcontinuity} we can replace $F(t, q)$ by
  \begin{equation}
    F(t,(x_1,\ldots,x_{n-1},L; \rho_0,\ldots,\rho_n))
  \end{equation}
  with pointwise error over $B$ bounded by $C_{\lambda,H'(P)}[w(t-s) +
  (t-s)]$.  For all but an $O((t-s)^2)$ portion of the $x$-volume, the
  value $t-s$ is the minimum above, and as before we obtain the
  indicated approximation of the integral over $B$.
\end{proof}

\begin{proof}[Proof of Lemma~\ref{lem:diffmu}]
  Essentially we are verifying the Leibniz rule, but we are unable to
  find a version of this to cite for kernels.  We first obtain
  quantitative control over our linear approximations of $f$ and $\ell$.
  Namely, fix any measurable $|J| \leq 1$.  We have 
  \begin{multline}
    \label{eqn:linearizedf}
    \int J(\rho_+) [f(t,\rho_-,d\rho_+) - f(s,\rho_-,d\rho_+) - (t-s)
    (\cL^{\kappa} f)(t,\rho_-,d\rho_+)]  \\
    = \int_s^t J(\rho_+) [(\cL^{\kappa} f)(\tau,\rho_-,d\rho_+) -
    (\cL^\kappa f)(t,\rho_-,d\rho_+)] d\tau.
  \end{multline}
  The difference of $\cL^\kappa f$ at different times can be expressed
  in terms of $f$ and $H$ again, 
  \begin{equation}
    \begin{aligned}
      &(\cL^\kappa f)(\tau, \rho_-, d\rho_+) - (\cL^\kappa
      f)(t,\rho_-,d\rho_+) \\ 
      &= \int (H[\rho_*,\rho_+]-H[\rho_-,\rho_*]) \left(\int_t^\tau
        (\cL^\kappa f)(\tau',\rho_-,d\rho_*) d\tau'\right)
      f(\tau,\rho_*,d\rho_+) \\
      &\quad + \int (H[\rho_*,\rho_+]-H[\rho_-,\rho_*])
      f(t,\rho_-,d\rho_*) \left(\int_t^\tau (\cL^\kappa
        f)(\tau',\rho_*,d\rho_+) \, d\tau'\right) \\
      &\quad - \left[\int_t^\tau \int \left(H[\rho_+,\rho_*]
          (\cL^\kappa)(\tau',\rho_+,d\rho_*)\right. \right. \\
      &\qquad \qquad \left. \left. \phantom{\int} - H[\rho_-,\rho_*]
          (\cL^\kappa)(\tau',\rho_-,d\rho_*)\right) \, d\tau'\right]
      f(\tau,\rho_-,d\rho_*) \\
      &\quad -\left[\int H[\rho_+,\rho_*] f(\tau,\rho_+,d\rho_*)
      \right. \\
      &\qquad \qquad \left. \phantom{\int} - H[\rho_-,\rho_*]
        f(\tau,\rho_-,d\rho_*)\right] \int_t^\tau 
      (\cL^\kappa f)(\tau',\rho_-,d\rho_+)
    \end{aligned}
  \end{equation}
  Noting that $\|\cL^\kappa f\| \leq 3 \lambda^2 H'(P)$, we integrate
  over $\rho_+$, then $\rho_*$, and find that \eqref{eqn:linearizedf}
  is bounded by
  \begin{equation}
    9 H'(P)^2 \lambda^3 (t-s)^2.
  \end{equation}
  Next, for any $|J| \leq 1$,
  \begin{multline}
    \label{eqn:linearizedh}
    \int J(\rho_0) [\ell(t,d\rho_0) - \ell(s,d\rho_0) - (t-s)(\cL^0
    \ell)(t,d\rho_0)] \\
    = \int_s^t J(\rho_0) [(\cL^0 \ell)(\tau, d\rho_0) - (\cL^0
    \ell)(t, d\rho_0)] d\tau,
  \end{multline}
  and 
  \begin{equation}
    \begin{aligned}
      (\cL^0 \ell)(\tau,d\rho_0) & - (\cL^0 \ell)(t,d\rho_0) \\
      &= \int H[\rho_*,\rho_0] \left(\int_t^\tau (\cL^0
        \ell)(\tau',d\rho_*) \, d\tau'\right) f(\tau,\rho_*,d\rho_0)
      \\
      &\quad + \int H[\rho_*,\rho_0] \ell(t,d\rho_*) \left(\int_t^\tau
        (\cL^\kappa f)(\tau', \rho_*, d\rho_0) \right) \\
      &\quad - \left[\int_t^\tau \left(\int H[\rho_0,\rho_*]
          (\cL^\kappa f)(\tau',\rho_0,d\rho_*)\right) d\tau' \right]
      \ell(\tau, d\rho_0) \\
      &\quad - \left[\int H[\rho_0,\rho_*] f(t,\rho_0,d\rho_*) \right]
      \int_t^\tau (\cL^0 \ell)(\tau',d\rho_0) \, d\tau'.
    \end{aligned}
  \end{equation}
  We have $\|\cL^0 \ell\| \leq 2 \lambda H'(P)$, so by integrating over
  $\rho_0$ and then $\rho_*$ we find \eqref{eqn:linearizedh} is
  bounded by 
  \begin{equation}
    5 H'(P)^2 \lambda^2 (t-s)^2.
  \end{equation}
  Returning to the problem of establishing our Leibniz rule, note that
  $e^{-\lambda L} \bOne_{\Delta^L_n}(dx)$ factors from both $\mu_n$
  and $\cL^* \mu_n$.  It will therefore suffice to obtain a bound on
  the $(\rho_0,\ldots,\rho_n)$ portion.

  We argue by induction.  In the case $n = 0$, we have only the
  difference in \eqref{eqn:linearizedh}, and the result holds.  Now
  suppose that the result holds for case $n$, and consider $n+1$.
  Choose as a test function $J(\rho_0,\ldots,\rho_{n+1})$, which is
  measurable and has $|J| \leq 1$, and integrate it against
  \begin{equation}
    \label{eqn:Leibniz}
    \begin{aligned}
      & \ell(t,d\rho_0) \prod_{j=1}^{n+1}
        f(t,\rho_{j-1},d\rho_j) - \ell(s,d\rho_0) \prod_{i=1}^{n+1}
        f(s,\rho_{j-1},d\rho_j) \\
      & \quad = \int \ell(t,d\rho_0) \prod_{j=1}^n
      f(t,\rho_{j-1},d\rho_j)
      [f(t,\rho_n,d\rho_{n+1}) - f(s,\rho_n,d\rho_{n+1})] \\
      & \quad \quad + \left[\ell(t,d\rho_0) \prod_{j=1}^n
        f(t,\rho_{j-1},d\rho_j) - \ell(s,d\rho_0) \prod_{j=1}^n
        f(s,\rho_{j-1},d\rho_j)\right] f(s,\rho_n,d\rho_{n+1}).
    \end{aligned}
  \end{equation}
  For each $\rho_0,\ldots,\rho_n$, the function
  $J(\rho_0,\ldots,\rho_{n+1})$ is bounded and measurable in
  $\rho_{n+1}$, so $f(t,\rho_n, d\rho_{n+1}) -
  f(s,\rho_n,d\rho_{n+1})$ can be replaced with 
  \begin{equation}
    (t-s) \int J(\rho_0,\ldots,\rho_{n+1}) (\cL^\kappa f)(t,\rho_n,
    d\rho_{n+1}) 
  \end{equation}
  plus an error no larger than $9 H'(P)^2 \lambda^3 (t-s)^2$, which
  when integrated over the remaining variables grows by a factor
  $\lambda^n$.

  In the third line of \eqref{eqn:Leibniz}, noting $\int
  J(\rho_0,\ldots,\rho_{n+1}) f(s,\rho_n,d\rho_{n+1})$ is bounded and
  measurable in $(\rho_0,\ldots,\rho_n)$, we apply the inductive
  hypothesis, replacing the bracketed difference by
  \begin{equation}
    (t-s) \left[\frac{(\cL^0 \ell)(t, d\rho_0)}{\ell(t,d\rho_0)} +
      \sum_{i=1}^n \frac{(\cL^{\kappa} f)(t, \rho_{i-1},
        d\rho_i)}{f(t, \rho_{i-1}, d\rho_i)}\right] 
    \ell(t,d\rho_0) \prod_{j=1}^n f(t, \rho_{j-1}, d\rho_j)
  \end{equation}
  plus an $o(t-s)$ error.  After doing this, again noting
  $J(\rho_0,\ldots,\rho_{n+1})$ is measurable in $\rho_{n+1}$ for each
  fixed $\rho_0,\ldots,\rho_n$, we replace $f(s,\rho_n,d\rho_{n+1})$
  with $f(t,\rho_n, d\rho_{n+1})$ at a cost of $3\lambda^2
  H'(P)(t-s)$, which gets multiplied by the other factor $(t-s)$.

  Adding the modified versions of these two lines of
  \eqref{eqn:Leibniz}, we find exactly the $\rho_0,\ldots,\rho_{n+1}$
  portion of $\cL^* \mu_{n+1}$ plus an $o(t-s)$ error.
\end{proof}



\bibliographystyle{spmpsci}%
\bibliography{scl-markov}%

\end{document}